\documentclass{amsart}
\usepackage{latexsym,mathtools}
\usepackage{amscd,amsfonts,amsmath,amssymb,amsthm}
\usepackage[normalem]{ulem}
\usepackage{comment}
\usepackage{dsfont}
\usepackage{tikz,tikz-cd}
\usepackage{enumitem}
\usepackage{aligned-overset}

\usepackage[T1]{fontenc}

\usepackage{bm}
\usepackage[colorlinks]{hyperref}
\usepackage{color}
\definecolor{darkgreen}{RGB}{55,138,0}

\hypersetup{citecolor = darkgreen}

\numberwithin{equation}{section}

\theoremstyle{plain}

\newtheorem{theorem}{Theorem}[section]
\newtheorem{lemma}[theorem]{Lemma}
\newtheorem{proposition}[theorem]{Proposition}
\newtheorem{corollary}[theorem]{Corollary}

\theoremstyle{definition}
\newtheorem{definition}[theorem]{Definition}
\newtheorem{example}[theorem]{Example}

\newtheorem{remark}[theorem]{Remark}
\newtheorem{question}[theorem]{Question}

\numberwithin{equation}{theorem} 

\newtheorem{thmintro}{Theorem}

\newtheorem*{theorem*}{Theorem}

\mathchardef\mhyphen="2D

\DeclareMathOperator\Aut{Aut}
\DeclareMathOperator\Der{Der}
\DeclareMathOperator\PDer{PDer}

\DeclareMathOperator\Hom{Hom}
\DeclareMathOperator\Id{Id}
\DeclareMathOperator\loz{loz}
\DeclareMathOperator\lcm{lcm}

\DeclareMathOperator\rk{rk}
\DeclareMathOperator\Spec{Spec}

\newcommand\NN{\mathbb N}

\newcommand\ZZ{\mathbb Z}
\newcommand\FF{\mathbb F}

\newcommand\cO{\mathcal O}

\newcommand\fg{\mathfrak g}

\newcommand{\balpha}{{\bm{\alpha}}}
\newcommand\bb{\mathbf b}
\newcommand\bc{\mathbf c}
\newcommand\be{\mathbf e}
\newcommand\boldm{\mathbf m}
\newcommand\bv{\mathbf v}
\newcommand\bw{\mathbf w}
\newcommand\bp{\mathbf p}

\newcommand\bu{\mathbf u}

\newcommand{\mathd}{\mathrm{d}}
\newcommand\bzero{\mathbf 0}

\newcommand\ch{\operatorname{char}}
\renewcommand\div{\mathrm{div}}

\newcommand\inv{^{-1}}

\newcommand\kk{\Bbbk}
\newcommand\tensor{\otimes}
\newcommand\PC{P_{\bc}}
\newcommand\pcnt{Z}
\newcommand\PN{\text{P}\mathcal{N}}

\newcommand\SP{S_{\bp}}
\newcommand\ZPC{ZP_{\bc}}

\newcommand\grp[1]{{\langle #1 \rangle}}

\begin{document}

\title[Log-ozone groups and centers of polynomial Poisson algebras]
{Log-ozone groups and centers of polynomial Poisson algebras}

\author{Kenneth Chan, Jason Gaddis, Robert Won, James J. Zhang}

\address{(Chan) Department of Mathematics, 
University of Washington, Box 354350, Seattle, Washington 98195, USA}
\email{kenhchan@math.washington.edu, ken.h.chan@gmail.com}

\address{(Gaddis) Department of Mathematics, 
Miami University, Oxford, Ohio 45056, USA} 
\email{gaddisj@miamioh.edu}

\address{(Won) Department of Mathematics,
The George Washington University, Washington, DC 20052, USA}
\email{robertwon@gwu.edu}

\address{(Zhang) Department of Mathematics,
University of Washington, Box 354350, Seattle, Washington 98195, USA}
\email{zhang@math.washington.edu}

\begin{abstract}
In previous work, the authors introduced the ozone group of an
associative algebra as the subgroup of automorphisms which fix the 
center pointwise. The authors studied PI skew polynomial algebras,
using the ozone group to understand their centers and to characterize
them among graded algebras.

In this work, we introduce and study the log-ozone group of a 
Poisson algebra over a field of positive characteristic.
The log-ozone group is then used 
to characterize polynomial Poisson 
algebras with skew symmetric structure.
We prove that unimodular Poisson algebras with skew symmetric
structure have Gorenstein centers. A related result is proved
for graded polynomial Poisson algebras of dimension three.
\end{abstract}

\subjclass[2020]{17B63, 17B40, 16W25, 16S36}


\keywords{Polynomial Poisson 
algebra, Poisson center, log-ozone group}

\maketitle

\section*{Introduction}
\label{xxsec0}

Let $\kk$ be a field, and for a large part of the 
paper, $\kk$ has positive characteristic. 
Throughout, all algebras are $\kk$-algebras.

Let $\bp = (p_{ij}) \in M_n(\kk^\times)$ be 
a multiplicatively antisymmetric matrix. The 
\emph{skew polynomial ring} (with parameters $\bp$) 
is defined as
\[ \SP = \kk\langle x_1,\hdots,x_n : x_jx_i 
= p_{ij} x_ix_j \rangle.\]
In previous work, the authors defined the 
\emph{ozone group} of $\SP$ as the set of (graded) 
automorphisms which fix the center of $\SP$ 
pointwise \cite{CGWZ2}. These automorphisms are 
called \emph{ozone automorphisms}. In \cite{CGWZ3}, 
the study of ozone groups was extended to any 
Artin--Schelter regular algebra satisfying a 
polynomial identity. The authors proved that the 
skew polynomial rings can be characterized as 
those PI Artin--Schelter regular algebras having 
abelian ozone groups of maximal order. In this 
work, we approach the problem of studying 
semiclassical limits of the skew polynomial rings 
$\SP$ and, more generally, polynomial Poisson 
algebras.

Recall that a \emph{Poisson algebra} is a 
commutative algebra $A$ along with a bracket 
$\{,\}:A \times A \to A$ such that $(A,\{,\})$ is 
a Lie algebra and $\{a,-\}$ is a derivation for 
each $a \in A$. The study of polynomial Poisson 
algebras is closely related to that of 
Artin--Schelter regular algebras. In \cite{TWZ}, 
Tang, Wang, and the fourth-named author defined 
the notion of an ozone derivation of a polynomial 
Poisson algebra $P$: a Poisson derivation 
$\partial$ of $P$ is called an 
\emph{ozone derivation} if $\partial(z) = 0$ for 
all $z$ in the Poisson center of $P$. However, 
one of our major goals is to develop an invariant 
of Poisson algebras which characterizes those with 
skew-symmetric structure. In positive 
characteristic, the notion of an ozone 
derivation in the sense of \cite{TWZ} may not be 
useful. For example, the above definition includes 
\emph{every} Hamiltonian derivation.

For PI Artin--Schelter regular algebras, an 
alternative characterization of ozone automorphisms 
is that they are precisely those automorphisms 
induced by conjugation by some normal regular 
element. Inspired by this characterization, given 
a regular Poisson normal element $f$ in a Poisson 
algebra, we define the \emph{log-ozone derivation 
relative to $f$} as $\delta_f = f\inv \{-,f\}$. 
In positive characteristic, we then define the 
\emph{log-ozone group} to be the (abelian) group 
generated by the derivations $\delta_f$, as $f$ 
runs over all regular Poisson normal elements.

In Section \ref{xxsec1} we give necessary 
background information on the Poisson algebras 
of interest. In particular, we recall the 
definition of the \emph{modular derivation} of 
a polynomial Poisson algebra (Definition 
\ref{xxdef1.4}). The modular derivation is 
closely related to the Nakayama automorphism of 
an Artin--Schelter regular algebra. We also 
provide some methods for computing the centers 
of certain Poisson algebras including tensor 
products and Poisson Ore extensions.

Section \ref{xxsec2} includes the formal 
definition of the log-ozone group (Definition 
\ref{xxdef2.3}), its basic properties, and 
some examples. One main result in this 
section is Theorem \ref{xxthmD}, stated below. 

In Section \ref{xxsec3} we study Poisson 
algebra structures which arise as 
semi-classical limits of the skew polynomial 
rings $\SP$. We define them here.

\begin{definition}
\label{xxdef0.1}
Let $\bc:=(c_{ij})_{n\times n}\in M_n(\kk)$ 
be a skew-symmetric matrix. The corresponding
\emph{skew-symmetric Poisson algebra}
$\PC$ is the algebra 
$\kk[x_1,\hdots,x_n]$ with Poisson bracket 
given by 
\[ \{ x_i,x_j\} = c_{ij} x_ix_j \quad
\text{for all $i<j$}.\]
The Poisson center of $\PC$ is denoted $\ZPC$.
\end{definition}

For the rest of the introduction we assume 
that $\kk$ is a field of 
characteristic $p>0$, and in Theorems 
\ref{xxthmB}, \ref{xxthmC}, and \ref{xxthmD} 
further assume that $\kk$ is algebraically 
closed. One basic question in this paper is 
to understand when the center $Z(P)$ of a Poisson 
algebra $P$ is Gorenstein (or has other 
favorable homological properties; see also Question 
\ref{xxque5.6}).

\begin{thmintro}[Theorem \ref{xxthm3.5}]
\label{xxthmA}
If $\PC$ is unimodular, then $\ZPC$ is 
Gorenstein.
\end{thmintro}

In Section \ref{xxsec4} we focus on 
unimodular polynomial Poisson algebras. 
We prove the following general theorem 
in dimension three which is the Poisson 
version of \cite[Theorem D]{CGWZ3}.

\begin{thmintro}[Theorem 
\ref{xxthm4.3}]
\label{xxthmB}
Assume $p>3$. Let $P$ be a graded polynomial 
Poisson algebra of dimension $3$. If $P$ is 
unimodular, then $\pcnt(P)$ is Gorenstein.
\end{thmintro}

Note that the conclusion of Theorem 
\ref{xxthmB} does not hold in the case 
that $p = 3$, see Example \ref{xxexa5.4}(2).  

\begin{definition}
\label{xxdef0.2}
A polynomial Poisson algebra 
$P = \kk[x_1,x_2,x_3]$ has 
\emph{Jacobian structure} if there is 
some $\Omega \in \kk[x_1,x_2,x_3]$ such 
that
\[
\{ x_1,x_2 \} = 
\frac{\partial \Omega}{\partial x_3}, 
\qquad \{ x_2,x_3 \} = 
\frac{\partial \Omega}{\partial x_1}, 
\qquad \{ x_3,x_1 \} = 
\frac{\partial \Omega}{\partial x_2}.
\]
In this case, we denote the Poisson 
algebra by $P_{\Omega}$ and call $\Omega$ 
the \emph{potential} of $P_{\Omega}$.
\end{definition}

It is well known in characteristic zero 
that all unimodular polynomial Poisson 
algebras in three variables have Jacobian structure. 
For characteristic $p>3$, this is still 
true and, when $\kk$ is algebraically 
closed, the classification of possible 
potentials $\Omega$ is identical (see 
Lemma~\ref{xxlem4.1}).

We determine the log-ozone group of polynomial 
Poisson algebras in three variables (see also Proposition~\ref{xxpro4.9}).

\begin{thmintro}[Theorem \ref{xxthm4.7}]
\label{xxthmC}
Suppose $p>3$. Let $P=\kk[x_1,x_2,x_3]$ be 
a graded polynomial Poisson algebra. Then 
$\loz(P) = 0$ if and only if $P$ is 
isomorphic to $P_{\Omega}$ where 
$\Omega = x_1^3$ or $\Omega$ is irreducible. 
In particular, if $\loz(P) = 0$, then $P$ 
is unimodular.
\end{thmintro}

Another basic question in this paper is 
whether or not $\loz(P)$ is finite, or more 
precisely, if $\loz(P)$ has a uniform bound 
as below.

\begin{question}[Question \ref{xxque5.2}]
\label{xxque0.3}
Let $P$ be a graded polynomial Poisson 
algebra. Is then $|\loz(P)|$ bounded above 
by $\rk_{Z(P)}(P)$?
\end{question}

A positive answer to the above question 
would be a nice result analogous to 
\cite[Theorem E]{CGWZ3}. In his Bourbaki 
seminar, Kraft considered the following 
problem as one of the eight challenging 
problems in affine algebraic geometry 
\cite{Kr96}.

\begin{question}
\label{xxque0.4}
Find an algebraic-geometric characterization 
of $\kk[x_1,\cdots,x_n]$.
\end{question}

We state a characterization result for 
the Poisson algebras $\PC$ in Definition 
\ref{xxdef0.1} that is the Poisson version 
of \cite[Theorem F]{CGWZ3}.

\begin{thmintro}[Theorem \ref{xxthm2.12}]
\label{xxthmD}
Suppose $A$ is a graded polynomial Poisson 
algebra. Then $A$ is isomorphic to 
$P_{\bc}$ for some $\bc$ if and only if 
$|\loz(A)|=\rk_{Z(A)}(A)$ 
and every $\delta \in \loz(A)$ is 
diagonalizable when it acts on $A_1$.
\end{thmintro}

If every $\delta \in \loz(A)$ is 
diagonalizable when it acts on $A_1$, we 
call $A$ \emph{inferable} (see 
Definition~\ref{xxdef2.9}). The above 
theorem does not hold if we remove the 
condition that $A$ is inferable (see 
Example \ref{xxexa5.4}(1)). It would be very 
interesting to understand which graded 
polynomial Poisson algebras $A$ satisfy 
$|\loz(A)|=\rk_{Z(A)}(A)$.

In the final section, we provide some
comments, examples, open questions, and 
further directions.

\subsection*{Acknowledgments} 
R. Won was partially supported by an 
AMS–Simons Travel Grant and Simons 
Foundation grant \#961085. J.J. Zhang was 
partially supported by the US National 
Science Foundation (No. DMS-2302087).

\section{Background and basic results}
\label{xxsec1}

Suppose that $G$ is an abelian (semi)group.
An associative algebra $A$ is 
\emph{$G$-graded} if there is a vector space 
decomposition $A = \bigoplus_{g \in G} A_g$ 
such that $A_g A_h \subset A_{g+h}$ for all 
$g,h \in G$ and that $1_A\in A_0$. An 
element $a \in A_g$ is called 
\emph{homogeneous} of degree $g$. We also 
use the notation $|a|$ for the degree of 
$a$. If $\fg$ is a Lie algebra, with 
bracket $[\, ,]$, then $L$ is 
\emph{$G$-graded} if there is a vector 
space decomposition $\fg = 
\bigoplus_{g \in G} \fg_g$ such that 
$[\fg_g, \fg_h] \subset \fg_{g+h}$ for 
all $g, h \in G$.

A \emph{Poisson algebra} $A$ is a 
commutative $\kk$-algebra equipped with a 
bilinear bracket $\{,\}: A \times A \to A$ 
such that $(A,\{,\})$ is a Lie algebra and 
$\{a,-\}:A \to A$ is a derivation for all 
$a \in A$. If $A = \kk[x_1,\dots, x_n]$ as 
an associative ring, then we say that $A$ 
is a \emph{polynomial Poisson algebra}.

The Poisson algebra $A$ is \emph{$G$-graded} 
by an abelian (semi)group $G$ if it is 
$G$-graded as both an associative algebra 
and as a Lie algebra (with the same vector 
space decomposition). We say that a 
Poisson algebra is \emph{graded} if it is 
graded by the natural numbers $\NN$. The 
Poisson algebras $\PC$ in 
Definition~\ref{xxdef0.1} and the Poisson 
algebras with Jacobian structure $P_{\Omega}$ 
in Definition~\ref{xxdef0.2} are graded 
polynomial Poisson algebras with the usual 
grading where $\deg(x_i) = 1$ for all $i$.

The \emph{Poisson center} of $A$, denoted 
$Z(A)$, is the set
\begin{equation}
\label{E1.0.1}\tag{E1.0.1}
Z(A) = \left \{ z \in A : \{z,a\}=0 
\text{ for all $a \in A$}\right\}.
\end{equation}

Let $A$ and $B$ be Poisson algebras with 
brackets $\{,\}_A$ and $\{,\}_B$, 
respectively. Then $A \tensor B$ is a 
Poisson algebra with bracket
\[
\{a \tensor b, c \tensor d\} 
= \{a,c\}_A \tensor bd + ac \tensor 
\{b,d\}_B.
\]

\begin{lemma}
\label{xxlem1.1}
Let $(A,\{,\}_A)$ and $(B,\{,\}_B)$ be 
Poisson algebras. Then $\pcnt(A \tensor B) 
= \pcnt(A) \tensor \pcnt(B)$.
\end{lemma}

\begin{proof}
The inclusion $\pcnt(A) \tensor \pcnt(B) 
\subset \pcnt(A \tensor B)$ is clear. Let 
$\sum a_i \tensor b_i \in \pcnt(A \tensor B)$ 
with $b_i$ linearly independent. If 
$x \tensor 1 \in A \tensor B$, then 
\[
0 = \left\lbrace \sum a_i \tensor b_i, x 
\tensor 1\right\rbrace 
= \sum \{ a_i,x\} \tensor b_i.\]
It follows that $\{ a_i,x\}=0$ so 
$a_i \in \pcnt(A)$ for all $i$. Similarly 
$b_i \in \pcnt(B)$ for all $i$. Thus the 
other inclusion holds.
\end{proof}

We now review several definitions 
regarding derivations of Poisson algebras.

\begin{definition}
\label{xxdef1.2}
Let $(A,\{,\})$ be a Poisson algebra. 
\begin{enumerate}[leftmargin=*]
\item[(1)] 
A \emph{Poisson derivation} of $A$ is 
a $\kk$-linear map $\delta: A \to A$ 
such that for all $a,b \in A$,
\[
\delta(ab) =\delta(a)b+a\delta(b) \quad
\text{and}\quad \delta(\{a,b\}) 
= \{\delta(a),b\} + \{a,\delta(b)\}.
\]
The first condition guarantees $\delta$ is 
an ordinary derivation of the (associative, 
commutative) algebra $A$. The set of all 
Poisson derivations of $A$ is denoted 
$\PDer(A)$. There is a Lie algebra structure 
on $\PDer(A)$ given by 
$[\delta, d] = \delta \circ d - d\circ\delta$ 
for $\delta, d \in \PDer(A)$.
\item[(2)] 
Let $A = \bigoplus_{g \in G} A_g$ be a 
$G$-graded Poisson algebra. For $g \in G$, 
a Poisson derivation $\delta$ is 
\emph{graded {\rm{(}}of degree $g${\rm{)}}} 
if $\delta(A_h) \subseteq A_{g + h}$ for 
all $h \in G$. The vector space of graded 
Poisson derivations of $A$ of degree $g$ is 
denoted $\PDer_g(A)$. The set of all 
graded Poisson derivations 
$\underline{\PDer}(A) 
= \bigoplus_{g \in G} \PDer_g(A)$ forms a 
graded Lie algebra. 
\end{enumerate}
\end{definition}

\begin{example}
\label{xxexa1.3}
Let $A$ be a graded Poisson algebra. The 
\emph{Euler derivation} $E: A \to A$ is 
defined by $E(a)=|a| a$ for all 
homogeneous $a \in A$. It is easy to see 
that $E \in \PDer_0(A)$.
\end{example}

\subsection{Unimodular Poisson algebras}
\label{xxsec1.1}
As another important example, we recall a 
specific Poisson derivation of a 
polynomial Poisson algebra, which plays a 
similar role as the Nakayama automorphism 
of a PI Artin--Schelter regular algebra. 

\begin{definition}
\label{xxdef1.4}
Suppose $P = \kk[x_1,\hdots,x_n]$ is a 
polynomial Poisson algebra with bracket 
$\{,\}$. The \emph{modular derivation of $P$} 
is the derivation $\phi$ of $P$ defined by
\[
\phi(f) = \sum_{j=1}^n 
\frac{\partial}{\partial x_j}\{x_j,f\} 
\qquad \text{for all $f \in P$.}
\]
If $\phi=0$, then $P$ is called 
\emph{unimodular}.
\end{definition}

\subsection{Poisson Ore extensions}
\label{xxsex1.2}
Let $(A,\{,\}_A)$ be a Poisson algebra 
and $\alpha$ a Poisson derivation of $A$. 
A derivation $\beta$ of $A$ is called a 
\emph{Poisson $\alpha$-derivation} of the 
Poisson algebra $A$ if
\[
\beta(\{a,b\}_A) = \{ \beta(a),b \}_A 
+ \{ a, \beta(b) \}_A + \alpha(a)\beta(b) 
- \beta(a)\alpha(b).
\]
When $\alpha = 0$, a Poisson 
$\alpha$-derivation is just a Poisson 
derivation of $A$. The \emph{Poisson Ore 
extension} associated to $(A,\alpha,\beta)$ 
is the polynomial extension $A[t]$ with 
bracket
\[
\{a,b\} = \{a,b\}_A
\quad\text{and}\quad 
\{a,t\} = \alpha(a)t + \beta(a)
\]
for all $a,b \in A$. We denote this by 
$A[t;\alpha,\beta]_P$.

We call a Poisson derivation $\alpha$ 
\emph{log-Hamiltonian} if there exists a 
unit $u \in A^{\times}$ such that 
$\alpha(a) u = \{u,a\}$ for all $a \in A$. 
As above, let $\alpha$ be a Poisson 
derivation of the Poisson algebra $A$. 
Let 
\[ A^\alpha = \{ a \in A : \alpha(a)=0\} 
\quad\text{and}\quad 
\pcnt(A)^\alpha = A^\alpha \cap \pcnt(A).\]
If $\alpha$ is log-Hamiltonian 
(with corresponding unit $u$), then 
$A[t;\alpha]_P = A[ut]$, since
\[
\{ a, ut \} 
= \{a,t\} u + \{a,u\}t
= \alpha(a)t u + \{ a,u \} t
= \{u,a\}t + \{a,u\}t = 0.
\]
For each $i \in {\ZZ}$, let 
\[
N(i\alpha) = \{ y \in A^\alpha : \{y,a\} 
= i\alpha(a)y \text{ for all } a \in A\}.
\]
The following is a Poisson version of 
\cite[Lemma 2.3]{GKM}.

\begin{lemma}
\label{xxlem1.5}
Let $A$ be a Poisson domain and let 
$B=A[t;\alpha,\beta]_P$ be a Poisson Ore 
extension. 
\begin{enumerate}
\item[(1)] 
\label{pcnt1} 
Suppose $\beta=0$. Then $\pcnt(B) = 
\bigoplus_{i \geq 0} N(i\alpha) t^i$.
\item[(2)] 
\label{pcnt2} 
Suppose $\beta\neq 0$.
If $A$ has trivial bracket and $\alpha=0$, 
then $\pcnt(B) = A^\beta$ if $\kk$ has 
characteristic zero and 
$\pcnt(B)=A^\beta[t^p]$ if $\kk$ has 
characteristic $p>0$.
\end{enumerate}
\end{lemma}

\begin{proof}
(1) Let $y \in N(i\alpha)$ for 
some $i \geq 0$. Then for $a \in A$, 
\[
\{ a, yt^i \}
= \{a,y\}t^i + \{a,t^i\}y 
= -i\alpha(a)yt^i + i\alpha(a)t^iy
= 0.
\]
Similarly, $\{t, y t^i\}=0$, whence 
$y t^i\in Z(B)$.

Conversely, suppose $\sum a_it^i \in \pcnt(B)$. 
Note that $B$ is an $\NN$-graded Poisson 
algebra (by $t$-degree). Using this grading, 
we see that $a_it^i \in \pcnt(B)$ for each 
$i$. Then 
\[ 0 = \{ a_it^i, t\} = \{a_i,t\} t^i 
= \alpha(a_i)t^{i+1}.\]
Consequently, $a_i \in A^\alpha$ for each 
$i$. Moreover, if $b \in A$, then
\[ 
0 = \{ a_it^i,b\} = \{a_i,b\} t^i + a_i \{t^i,b\}
= \left( \{a_i,b\} - i\alpha(b)a_i \right)t^i
\]
which implies that $a_i \in N(i\alpha)$.

(2)
Let $z = \sum_{i=0}^n a_i t^i \in \pcnt(B)$
with $a_n \neq 0$. Then
\begin{align*}
0   = \{ z,t \} 
    = \sum \{ a_i, t \} t^i
    = \sum \beta(a_i) t^i.
\end{align*}
So, $\beta(a_i)=0$ for each $i$.
That is, $a_i \in A^\beta$ for each $i$.

Now, let $b \in A$ such that $\beta(b)\neq 0$. Then
\begin{align*}
0  = \{ z, b\} 
= \sum \left( \{a_i,b\}t^i  + a_i\{t^i, b\} \right)
= \sum \left( - i a_i \beta(b) t^{i-1} \right).
\end{align*}
It follows that $i\equiv 0 \mod \ch(\kk)$.
\end{proof}

Let $A$ be a graded Poisson algebra on 
$\kk[x_1,x_2]$. Then, up to isomorphism, 
either $A=P_\bc$ for some $\bc$ or else 
$A$ has the bracket given by 
$\{x_1,x_2\}=x_1^2$. We apply the 
preceding lemma to compute the center 
of the second type, which is sometimes 
called the Poisson Jordan plane.

\begin{example}
\label{xxexa1.6}
Suppose $\kk$ is a field of characteristic $p>0$.
Let $A=\kk[x_1,x_2]$ with Poisson bracket 
given by $\{x_1,x_2\}=x_1^2$.
This is a Poisson Ore extension 
$\kk[x_1][x_2;0,\beta]$ with $\beta(x_1)=x_1^2$.
By Lemma \ref{xxlem1.5}(2), $\pcnt(A)=\kk[x_1]^\beta[x_2^p]$.
Let $f = \sum_{i=0}^n \alpha_i x_1^i \in 
\kk[x_1]^\beta$. Then
\[
0   = \beta(f) 
= \sum_{i=0}^n \alpha_i \beta(x_1^i)
= \sum_{i=1}^n \alpha_i i x_1^{i+1}.
\]
It follows that $\alpha_i \neq 0$ implies 
that $i \equiv 0 \mod p$, and further 
$\pcnt(A)=\kk[x_1^p,x_2^p]$.
\end{example}

\section{The log-ozone group}
\label{xxsec2}

We next proceed to formally introduce 
the log-ozone group of a Poisson algebra. 
As mentioned in the introduction, 
the log-ozone group is defined in terms of 
Poisson normal elements, and so we recall 
this definition now.

\begin{definition}
\label{xxdef2.1}
Let $(A,\{,\})$ be a Poisson algebra. 
\begin{enumerate}[leftmargin=*]
\item[(1)]
An element $f \in A$ is called 
\emph{Poisson normal} if $\{a, f\} \in fA$ 
for all $a \in A$. The set of Poisson 
normal elements in $A$ is denoted $\PN(A)$.
\item[(2)]
Given any regular $f \in \PN(A)$, the 
\emph{log-ozone derivation} associated to 
$f$ is $\delta_f := f\inv \{-, f\}$. It is 
easy to see that $\delta_f \in \PDer(A)$. 
Note that if, further, $A$ is graded and 
$f$ is homogeneous, then $\delta_f \in 
\PDer_0(A)$.
\end{enumerate}
\end{definition}

If $f$ is a unit, then the log-ozone 
derivation $\delta_f$ is also log-Hamiltonian, 
and so log-ozone derivations generalize 
log-Hamiltonian derivations. The 
log-Hamiltonian derivations of a Poisson 
algebra have been studied by several authors 
\cite{LWW3, LWZ17, TWZ}.

\begin{lemma}
\label{xxlem2.2}
Let $(A, \{,\})$ be a Poisson algebra and 
let $f, g \in A$ be regular elements.
\begin{enumerate}
\item[(1)] 
If $f, g \in \PN(A)$, then $fg \in \PN(A)$ 
and $\delta_{fg} = \delta_f + \delta_g$.
\end{enumerate}
Further suppose that $A$ is a UFD.
\begin{enumerate}
\item[(2)] 
If $\gcd(f,g) = 1$ and $fg \in \PN(A)$, 
then $f, g \in \PN(A)$.
\item[(3)] 
If $fg \in \PN(A)$, then $f \in \PN(A)$ 
if and only if $g \in \PN(A)$.
\item[(4)] 
If $f,g \in \PN(A)$, then $\lcm(f,g), 
\gcd(f,g) \in \PN(A)$. 
\item[(5)] 
Suppose $A = \kk[x_1, \dots, x_n]$ and the 
Poisson structure on $A$ is graded. If 
$f,g \in \PN(A)$, then $\{f,g\} = q fg$ for 
some $q \in \kk$ and $\delta_f\delta_g=
\delta_g\delta_f$.
\item[(6)] 
Suppose $A$ is a $G$-graded Poisson algebra 
where $G$ is an ordered abelian group. If 
$f \in \PN(A)$ is written as a sum of 
homogeneous terms $f = f_1 + \cdots + f_m$ 
of distinct degrees, then each $f_i \in 
\PN(A)$ and $\delta_{f_i} = \delta_{f}$.
\end{enumerate}
\end{lemma}

\begin{proof}
(1) Suppose $f,g \in \PN(A)$. Then for any 
$a \in A$, 
\[ \{a, fg\} = f\{a, g\} + g\{a, f\} 
= fg \delta_g(a) + fg \delta_f(a) 
= fg(\delta_f(a) + \delta_g(a)).\] 
The result follows.

Now suppose that $A$ is a UFD.

(2) We prove the contrapositive. Suppose 
that $\gcd(f,g) = 1$. Without loss of 
generality, suppose that $f \not \in \PN(A)$ 
so that there exists some $a \in A$ such 
that $f \nmid \{a,f\}$. Note that 
$\{a, fg\} = f\{a,g\} + g\{a, f\}$ and 
since $\gcd(f,g) = 1$, we have 
$f \nmid g \{a,f\}$ and hence 
$f \nmid \{a, fg\}$. Therefore, 
$fg \not \in \PN(A)$.
    
(3) Suppose $fg \in \PN(A)$ and $f \in \PN(A)$. 
Then for all $a \in A$, we have 
$f\{a, g\} = \{a, fg\} - g \{a, f\}$, and 
since $fg$ divides the right-hand side, so 
$fg \mid f\{a,g\}$ and hence $g \mid \{a,g\}$.

(4) Suppose $f,g \in \PN(A)$. Suppose that 
$m$ is any multiple of $f$, say $m = fd$ for 
some $d \in A$. Then for any $a \in A$, we have 
$\{a, m\} = \{a, fd\} = f\{a, d\} + d\{a,f\}$ 
and so $f \mid \{a, m\}$. Similarly, if $n$ is 
any multiple of $g$, then $g \mid \{a, n\}$. 
In particular, we have $f, g \mid 
\{a, \lcm(f,g)\}$, and so $\lcm(f,g) \mid 
\{a, \lcm(f,g)\}$. Therefore, $\lcm(f,g) \in 
\PN(A)$. Now since $fg = \lcm(f,g) \gcd(f,g)$, 
by parts (1) and (3), we have $\gcd(f,g) \in 
\PN(A)$.

(5) Suppose now that $A = \kk[x_1, \dots, x_n]$ 
with $\{, \}$ graded and $f,g \in \PN(A)$. 
First note that if $\gcd(f,g) = 1$, then since 
$\{f,g\}$ is a common multiple of $f$ and $g$ 
of degree $\deg f + \deg g$, therefore 
$\{f,g\} = qfg$ for some $q \in \kk$.

In general, let $\lcm(f,g)$ be $f_1^{n_1}
\cdots f_{s}^{n_s}$ where $f_i$ are distinct 
irreducible elements and $n_i>0$ for 
$i=1,\cdots,s$. By parts (2,3), each $f_i$ 
is Poisson normal, and by the previous 
paragraph, $\{f_i,f_j\}=d_{ij} f_i f_j$ for 
some $d_{ij}\in \kk$ (with $d_{ii}=0$).
Now $f=f_1^{a_1}\cdots f_{s}^{a_s}$
and $g= f_1^{b_1}\cdots f_{s}^{b_s}$. An 
easy computation shows that $\{f,g\}=d fg$ 
where $d=\sum_{i,j=1}^{s} a_i b_j d_{ij}$.

For the last claim, let $a\in A$. Then by 
the Jacobi identity
\begin{align*}
0 &= \{f,\{g,a\}\} + \{f,\{a,g\}\} 
   + \{a,\{f,g\}\} \\
&= \{ f, \delta_g(a)g\} - \{g, \delta_f(a)f\} 
   + q \{a, fg\} \\
&= \delta_f\delta_g(a) fg + q \delta_g(a) fg
   - \delta_g\delta_f(a) fg + q \delta_f(a) fg
   - q \delta_f(a) fg - q \delta_g(a) fg \\
&= (\delta_f\delta_g(a) - \delta_g\delta_f(a)) fg.
\end{align*}
Thus, $\delta_f\delta_g=\delta_g\delta_f$.

(6) Suppose that $a \in A$ is homogeneous. 
Then
\[
\sum_{i=1}^m \{f_i, a\} = \{f,a\} 
= \delta_f(a)f = \delta_f(a)\sum_{i=1}^m f_i.
\]
Since the $f_i$ have distinct degrees, the 
above equations forces that $\delta_{f}(a)$ 
is homogeneous. Consequently, $\{f_i, a\} 
= \delta_f(a) f_i$ for all $i$. Since this 
is true for every homogeneous $a \in A$, 
therefore $f_i \in \PN(A)$ and 
$\delta_{f_i} = \delta_f$.
\end{proof}

Lemma~\ref{xxlem2.2}(1) shows that the set 
of log-ozone derivations has an abelian 
monoid structure. In the case that $\kk$ 
has positive characteristic, then there is, 
in fact, an abelian group structure.

\begin{definition}
\label{xxdef2.3}
Let $A$ be a Poisson algebra over a field 
$\kk$ of characteristic $p > 0$. The 
\emph{log-ozone group} of $A$, denoted 
$\loz(A)$, is defined to be the set of 
log-ozone derivations of $A$. That is,
\[ \loz(A) = \{ \delta_f \mid f \in \PN(A) 
\text{ is regular}\}.\]
\end{definition}

By Lemma~\ref{xxlem2.2}(6), if $A$ is a 
$\ZZ^n$-graded domain and $\delta_f \in 
\loz(A)$, then $f$ may be assumed to be 
$\ZZ^n$-homogeneous.

Henceforth, assume that $\kk$ is a field of 
characteristic $p>0$, so that in all 
subsequent examples, the log-ozone group is, 
in fact, a group. The next example shows 
that the log-ozone group need not be finite.

\begin{lemma}
\label{xxlem2.4}
Let $A$ be a Poisson algebra.
\begin{enumerate}
\item[(1)] 
$\loz(A)$ is an additive subgroup of $\PDer(A)$. 
Consequently, $\loz(A)$ is a $\ZZ/(p)$-vector 
space.
\item[(2)]
If $|\loz(A)|$ is finite, then $|\loz(A)| = 
p^{\ell}$ for some integer $\ell \geq 0$. 
\item[(3)] 
If $A$ is connected $\NN$-graded, then $\loz(A)$ 
is a subgroup of $\PDer_0(A)$ and a subgroup of 
$HP^1(A)$, the first Poisson cohomology group.
\end{enumerate}
\end{lemma}

\begin{proof}
As mentioned previously, Lemma~\ref{xxlem2.2}(1) 
shows that there is an abelian monoid structure 
on $\loz(A)$. Since, for any regular element 
$f \in \PN(A)$, we have $\delta_{f^p} = 0$, 
therefore $\delta_{f^{p-1}} = - \delta_f$ and 
so $\loz(A)$ is an abelian group. This proves 
(1) and (2) is clear.

For (3), suppose $A$ is connected $\NN$-graded. 
If $f$ is homogeneous, it is clear that $\delta_f 
\in \PDer_0(A)$. By Lemma~\ref{xxlem2.2}(6), 
every element of $\loz(A)$ is of the form 
$\delta_f$ for some homogeneous $f$, and so 
$\loz(A) \subseteq \PDer_0(A)$. The first 
Poisson cohomology $HP^1(A)$ is defined  to be 
$\PDer(A)/\{\text{Hamiltonian derivations}\}$. 
In degree $0$, there are no nonzero 
Hamiltonian derivations (since $A_0=\kk$), 
so $\loz(A)$ is also a subgroup of $HP^1(A)$.
\end{proof}

\begin{example}
\label{xxexa2.5}
Let $A = \kk[x_1, x_2, \dots]$ be a polynomial 
ring in infinitely many variables with 
skew-symmetric Poisson structure. That is, 
for each $i, j \in \NN$ choose 
$c_{ij} = - c_{ji} \in \kk$ and set 
$\{x_i, x_j\} = c_{ij} x_i x_j$. Further 
suppose the $c_{ij}$, for $i\neq j$, 
are all nonzero. Then 
$\delta_{x_i} \neq \delta_{x_j}$ for any 
$i \neq j$ (since they differ on $x_i$). 
Hence, $\loz(A)$ is infinite.
\end{example}

In this paper, we will focus mostly on 
polynomial Poisson algebras (in finitely 
many variables). We compute a first example 
of the log-ozone group of such a Poisson 
algebra.

\begin{example}
\label{xxexa2.6}
Suppose $p \neq 2$ and fix $c \in 
\kk^\times$. Let $P = \kk[x_1,x_2]$ be 
the polynomial Poisson algebra with 
Poisson bracket $\{x_1, x_2\} = c x_1 x_2$. 
Suppose that $x_1^ix_2^j \in \pcnt(P)$. Then
\[ 
0 = \{ x_1, x_1^ix_2^j \} 
= x_1^i \{ x_1, x_2^j \} 
= j x_1^i x_2^{j-1} \{x_1,x_2\}
= j c x_1^{i+1} x_2^j.\]
Hence, $p \mid j$. A similar argument shows 
$p \mid i$. This proves that 
$\pcnt(P)=\kk[x_1^p,x_2^p]$.

On the other hand, it is not difficult to 
verify that $\loz(P)$ is generated by 
$\delta_{x_1}$ and $\delta_{x_2}$ and 
$|\loz(P)|=p^2=\rk_{Z(P)}(P)$. Also note that 
$Z(P)=\ker \delta_{x_1}\cap \ker \delta_{x_2}$. 
\end{example}

Examples \ref{xxexa1.6} and \ref{xxexa2.6} 
show that the center of a graded Poisson 
algebra on $\kk[x_1,x_2]$ is necessarily 
regular. 

We now introduce some basic results on the 
log-ozone group, including some techniques 
for computing it in certain examples. The 
next proposition is a Poisson version of 
\cite[Lemma 1.12]{CGWZ3}.

\begin{proposition}
\label{xxpro2.7}
Let $(A,\{,\}_A)$ and $(B,\{,\}_B)$ be 
connected graded Poisson domains. Then 
$\loz(A \tensor B) \cong \loz(A) \times \loz(B)$. 
In particular, $\loz(A[t]) \cong  \loz(A)$.
\end{proposition}

\begin{proof}
If $u \in \PN(A)$, then $u \tensor 1 \in 
\PN(A \tensor B)$, so $\delta_{u\otimes 1}
=\delta_{u}\otimes Id_{B}\in 
\loz(A \tensor B)$. A similar argument 
with $\PN(B)$ makes it clear $\loz(A) 
\times \loz(B) \subset \loz(A \tensor B)$. 
Now suppose that $\partial:=\delta_{f} 
\in \loz(A \tensor B)$ where $0\neq f=
\sum_{i=1}^{n} a_i \otimes b_i\in A\tensor B$
for some $n\geq 1$. By Lemma \ref{xxlem2.2}(6), 
we can assume that $a_i$ are homogeneous of 
the same degree and that $b_i$ are homogeneous 
of the same degree. We may further assume that 
$\{a_1,\cdots, a_n\}$ (resp. 
$\{b_1, \cdots, b_n\}$) are linearly 
independent. For any $x\in A$,
because $\partial \in \PDer_0(A \tensor B)$, then 
$\partial(x \tensor 1) \in A \tensor B_0 
= A \tensor \kk$. Thus, $\partial$ restricts
to a Poisson derivation of $A$ which is denoted 
by $\partial_1$. Hence
\[\left(\sum_{i=1}^{n} a_i\tensor b_i\right)
(\partial_1(x)\tensor 1)=f \delta_{f}(x\tensor 1)=
\{f, x\tensor 1\}=\sum_{i=1}^{n}\{a_i, x\}
\otimes b_i,\]
which implies that $\partial_1(x)=\delta_{a_i}
\in \loz(A)$ for every $i$. Similarly, the 
restriction of $\partial$ onto $B$, denoted 
by $\partial_2$, is $\delta_{b_i}\in \loz(B)$ 
for all $i$. This implies that
$\partial=\delta_{a_i}\otimes \Id_{B}+ 
\Id_{A}\otimes  \delta_{b_j}$ which is identified 
with $(\delta_{a_i},\delta_{b_j})
\in\loz(A)\times \loz(B)$ for any 
pair $(i,j)$. 
\end{proof}

The next result concerns Poisson Ore extensions 
as defined in Section \ref{xxsex1.2}.

\begin{lemma}
\label{xxlem2.8}
Let $A$ be a Poisson domain and let 
$B=A[t;\alpha,\beta]_P$ be a Poisson Ore 
extension.
\begin{enumerate}
\item[(1)]  
If $\beta=0$, then $\loz(B)$ is generated as 
group by $\alpha$ and
\[ \{ \delta_{f} : \delta_f \in \loz(A) 
\text{ and } f \text{ divides } \alpha(f)\}.\]
\item[(2)]  
If $\alpha=0$, then 
\[ \loz(B) = \{ \delta_{f} : \delta_f \in 
\loz(A) \text{ and } f \text{ divides } 
\beta(f)\}.\]
\end{enumerate}
\end{lemma}

\begin{proof}
(1) By abuse of notation, extend $\alpha$ to a 
derivation of $B$ by defining $\alpha(t) = 0$. 
Then it is clear that $\alpha \in \loz(B)$ 
since $\alpha=\delta_t$. Now suppose $\delta_f 
\in \loz(A)$ for some $f$ which divides 
$\alpha(f)$. Let $\sum a_i t^i \in B$. Then
\begin{align*}
\left\lbrace \sum a_i t^i, f \right\rbrace 
&= \sum \left( \{ a_i, f\}t^i + a_i\{t^i, f\} 
   \right) \\
&= \sum \left( \delta_f(a_i) f - i a_i 
   \alpha(f) \right) t^i\\
&= f \sum \left( \delta_f(a_i) t^i 
  - i a_i f\inv \alpha(f) t^{i} \right).
\end{align*}
Thus, $f \in \PN(B)$ and so $\delta_f \in \loz(B)$.

Conversely, let $\partial \in \loz(B)$ and 
suppose $\partial = \delta_g$ where $g \in \PN(B)$.
By Lemma~\ref{xxlem2.2}(6), we may assume that 
$g = ft^i$ for some $f \in A$ and $i \in \NN$. If 
$r \in A$, then 
\[ \partial(r) f t^i
= \partial(r) g
= \{ r,g \} 
= \{ r, f t^i \}
= \left( \{r,f\} + i f \alpha(r) \right) t^i.
\]
Consequently, $\partial(r) f = \{r,f\} + i f 
\alpha(r)$. Since $f$ divides the left-hand side, 
then $f$ divides $\{r, f\}$. That is, $f \in \PN(A)$ 
and $\partial(r) = \delta_f(r) + i\alpha(r)$. 
Similarly,
\[ \partial(t) f t^i 
= \partial(t) g
= \{ t, g \} 
= \{ t, f t^i \}
= -\alpha(f) t^{i+1}.
\]
Thus $\partial(t) = - f\inv \alpha(f) t$. It 
follows that $f$ divides $\alpha(f)$. Now notice 
that $\delta_f(t) = f\inv \{f, t\} = -f\inv 
\alpha(f) t$ and so $\partial 
= \delta_f + i \alpha$. 
The result follows.

(2) Suppose $f \in \PN(A)$ and $f\mid \beta(f)$. 
Let $\sum a_i t^i \in B$. Then
\begin{align*}
\left\lbrace \sum a_i t^i, f \right\rbrace 
= \sum \left( \{ a_i, f\}t^i + a_i\{t^i, f\} \right)
= f \sum \left( \delta_f(a_i) t^i - i a_i f\inv 
\beta(f) t^{i-1} \right).
\end{align*}
Thus, $f \in \PN(B)$ and so $\delta_f \in \loz(B)$.
Conversely, let $\partial \in \loz(B)$ induced by
$f = \sum_{i=0}^k f_i t^i \in \PN(B)$ with 
$f_k \neq 0$. If $a \in A$, then
\[ \partial(a) \sum_{i=0}^k f_i t^i 
= \{ a,f \} 
= \sum_{i=0}^k \{ a, f_i t^i \}
= \sum_{i=0}^k \left( \{a,f_i\} t^i 
+ i f_i \beta(a) t^{i-1} \right).
\]
Comparing coefficients of $t^k$ gives 
that $\partial(a)\in A$ and that
$\partial(a)f_k = \{a,f_k\}$, so 
$\partial(a) = \delta_{f_k}(a)$.
Similarly,
\[ \partial(t) \sum_{i=0}^k f_i t^i 
= \{ t,f \} 
= \sum_{i=0}^k \{ t, f_i t^i \}
= \sum_{i=0}^k -\beta(f_i) t^i.
\]
Again comparing coefficients of $t^k$ gives 
$\partial(t) f_k = -\beta(f_k)$, so 
$\partial(t) = -f_k\inv \beta(f_k) \in A$. 
Hence, $\partial = \delta_{f_k}$
and $f_k \mid \beta(f_k)$.
\end{proof}

\subsection{The structure of the log-ozone group}
\label{xxsec2.1}

Let $A$ be a Poisson domain.
Let $Z$ denote its Poisson center. For each $\delta\in \loz(A)$ 
define 
\[ 
A_{\delta} = \{0\} \cup \{f \in A \mid 
\delta_f = \delta\}.
\]
Let 
\[
A_{\loz} = \sum_{\delta \in \loz(A)} A_{\delta} 
\subseteq A,
\]
where the sum occurs inside of $A$, viewed as a 
$Z$-module. It is clear that $A_0 = Z$ and that 
each $A_{\delta}$ is a $Z$-submodule of $A$.

\begin{definition}
\label{xxdef2.9}
Let $A$ be a Poisson domain.
\begin{enumerate}
\item[(1)]
We say $A$ is {\it inferable} if every 
$\delta\in \loz(A)$ is diagonalizable when 
acting on $A$.
\item[(2)]
We say $A$ is {\it quasi-inferable} if 
every nonzero $\delta\in \loz(A)$ has a 
nonzero eigenvalue when acting on $A$. 
\item[(3)]
We say $A$ is {\it loz-decomposable} if 
$A_{\loz} = \bigoplus_{\delta \in \loz(A)} 
A_{\delta}$.
\end{enumerate}
\end{definition}

\begin{lemma}
\label{xxlem2.10}
Assume $\kk$ is algebraically closed.
Suppose that $A$ is a graded polynomial 
Poisson algebra $\kk[x_1,\cdots,x_n]$.
\begin{enumerate}
\item[(1)]
If $A$ is inferable, then it is 
quasi-inferable.
\item[(2)]
If $A$ is quasi-inferable, then it is 
loz-decomposable.
\end{enumerate}
\end{lemma}

\begin{proof}
(1) This is clear since a diagonalizable 
nonzero linear map has at least one nonzero 
eigenvalue. 

(2) If $\loz(A)$ is trivial, the result 
is clear, so we may assume that $\loz(A)$ 
is nontrivial. It suffices to show that a 
sum of nonzero elements from distinct 
$A_{\delta}$ is nonzero. Suppose, for 
contradiction, that this is not the case, 
and choose a minimal $n$ such that we may 
choose distinct 
$\delta_1, \dots, \delta_n \in \loz(A)$ 
and nonzero $f_i \in A_{\delta_i}$ so that 
$\sum_{i = 1}^n f_i = 0$.

Write
\[
-f_1 = \sum_{i=2}^{n} f_i.
\]
After multiplying by $f_1^{p-1}$, we may 
assume that $f_1$ is Poisson central and 
none of $f_2, \dots, f_{n}$ are Poisson 
central.

By Lemma~\ref{xxlem2.2}(5), $\delta_i 
\delta_j = \delta_j \delta_ i$ for all 
$2\leq i,j\leq n$ (with $\delta_{1}=0$). 
Since $A$ is quasi-inferable, $\delta_n$ 
has a nonzero eigenvalue $\lambda_n$. In 
fact, since $A$ is graded, each $\delta_i$ 
acts on $A_1$ and we may assume that 
$\lambda_n$ is an eigenvalue 
$\delta_n|_{A_1}$ (since if $\delta_n|_{A_1}$ 
has all zero eigenvalues, then it will 
have zero eigenvalues on all of $A$). Let 
$V\subseteq A_1$ be the $\lambda_n$ 
eigenspace. Since the $\delta_i$'s commute, 
they all act on $V$. Viewing the $\delta_i$ 
as linear maps $V \to V$, since we are 
assuming $\kk$ is algebraically closed, we 
see that they are simultaneously 
triangularizable. In particular, they all 
share a common eigenvector $x\in V$. More 
precisely, we have $\delta_i(x)=\lambda_i x$ 
for $i=1,\ldots,n$ for some 
$\lambda_1,\ldots,\lambda_{n-1}\in\kk$, 
and $\lambda_n\in \kk^\times$ (as chosen 
above). Then
\[
0 = \left\{ -f_1, x\right\} 
= \sum_{i=2}^{n} \left\{ f_i, x \right\} 
= \sum_{i=2}^{n} -f_i \delta_i(x) 
= \sum_{i=2}^{n} -\lambda_i f_i x.
\]
Since $A$ is a domain, $\sum_{i=2}^{n} 
\lambda_i f_i=0$. Since $\lambda_n \neq 0$, 
this contradicts the minimality of $n$.
\end{proof}

Example~\ref{xxexa2.14} below shows that 
a graded polynomial Poisson algebra may be 
loz-decomposable but not quasi-inferable.

\begin{lemma}
\label{xxlem2.11}
Suppose $A$ is a graded polynomial Poisson 
algebra $\kk[x_1, \dots, x_n]$ and assume 
that $A$ is loz-decomposable. Then
\begin{enumerate}
\item[(1)] 
$A_{\loz}$ is a subalgebra of $A$, so 
$A_{\loz}$ is a $\loz(A)$-graded $\kk$-algebra.
\item[(2)] 
For each $\delta \in \loz(A)$, $A_{\delta}$ 
is a rank one $Z$-module. Hence, 
$|\loz(A)|=\rk_{Z}(A_{\loz}) \leq \rk_Z(A)$.
\item[(3)] 
If $A$ is generated as an algebra by Poisson 
normal elements, then $|\loz(A)| = \rk_Z(A)$.
\item[(4)]
$\rk_{Z}(A) = p^m$ for some $m \leq n$.
\end{enumerate}
\end{lemma}

\begin{proof}
(1) This follows immediately from 
Lemma~\ref{xxlem2.2}(1,5).

(2) Let $Q = Z(Z\setminus\{0\})\inv$ be the 
field of fractions of $Z$. Fix 
$\delta \in \loz(A)$ and let 
$A^Q_{\delta} : = A_{\delta} \otimes_{Z} Q$ 
so $\rk_Z(A_{\delta}) = \rk_Q(A_{\delta}^Q)$. 
For any $f,g \in A_{\delta}$, 
$f^{p-1}g \in A_0 = Z$. Hence, 
$f^p(f^{p-1}g)\inv = g\inv f \in Q$. 
Therefore, $A_{\delta}^Q$ is a rank one 
$Q$-module, as desired. Then 
$|\loz(A)| = \rk_Z(A_{\loz})$ follows 
the definition of loz-decomposability.

(3) If $A$ is generated by Poisson normal 
elements, then $A = A_{\loz}$ and since 
$A = A_{\loz} = 
\bigoplus_{\delta \in \loz(A)} A_{\delta}$ 
is a direct sum of rank one $Z$-modules, 
the result follows.

(4) Let $C = \kk[x_1^p,\ldots, x^p_n]$. 
It is easy to see that $C$ is a Poisson 
central subalgebra so 
$C \subseteq Z \subseteq A$ and that 
$\rk_C(A) = p^n$. Since 
$\rk_C(A) = \rk_C(Z) \cdot \rk_Z(A)$, 
the result follows.
\end{proof}

Finally, we show that the Poisson 
algebras $P_{\bc}$ can be characterized 
as those graded Poisson algebras which 
have log-ozone groups of maximal order.

\begin{theorem}
\label{xxthm2.12}
Suppose $\kk$ is algebraically closed and suppose $A$ is a 
graded polynomial Poisson algebra 
$\kk[x_1,\cdots,x_n]$. Then the following 
are equivalent:
\begin{enumerate}
\item[(a)]
$A$ is isomorphic to $P_{\bc}$ for some $\bc$.
\item[(b)]
$|\loz(A)|=\rk_{Z}(A)$ and $A$ is inferable.
\item[(c)]
$|\loz(A)|=\rk_{Z}(A)$ and $A$ is quasi-inferable.
\item[(d)]
$|\loz(A)|=\rk_{Z}(A)$ and $A$ is loz-decomposable.
\end{enumerate}
\end{theorem}

\begin{proof}
(a) $\Rightarrow$ (b)
Suppose $A=P_{\bc}$. Let $\delta_{f}\in \loz(A)$ for some 
Poisson normal element $f$.
By Lemma \ref{xxlem2.2}(6), $\delta_{f}=\delta_{f_i}$ for 
a $\ZZ^n$-homogeneous
element $f_i$, so we may assume that $f_i$ is a monomial. 
In this case, it is easy to check that $\delta_{f_i}$ is diagonalizable. Therefore $A$ is inferable. Since $A$ is generated by Poisson normal elements $\{x_i\}_{i=1}^{n}$, by Lemma \ref{xxlem2.11}(3), $|\loz(A)|=\rk_Z(A)$.

(b) $\Rightarrow$ (c) See Lemma \ref{xxlem2.10}(1).

(c) $\Rightarrow$ (d) See Lemma \ref{xxlem2.10}(2).

(d) $\Rightarrow$ (a)
We assume that $|\loz(A)|=\rk_{Z}(A)$. 
Recall from the definition that
\[ A_{\loz} = \bigoplus_{\delta\in\loz(A)}A_{\delta}\]
where each $A_{\delta}$ has rank 1 over $Z$ by Lemma \ref{xxlem2.11}(2). So
\[ \rk_{A_{\loz}}(A) = \rk_Z(A) - \rk_Z(A_{\loz}) = 
\rk_{Z}(A)-|\loz(A)|=0.\]
Then $Q(A_{\loz}) = Q(A)$, which implies that
\[ Q(A)= Q(A_{\loz}) = \bigoplus_{\delta \in \loz(A)} A_{\delta} Q(Z).\]
Set $Q_\delta = A_{\delta} Q(Z)$ (in particular $Q_0 = Q(Z)$), so that 
$Q(A) 
= \bigoplus_{\delta\in \loz(A)} Q_{\delta}$. 
If $\delta \in \loz(A)$, then $\delta$ extends to a Poisson derivation of $Q: =Q(A)$ uniquely. So, by abuse of notation, we may assume $\delta\in \PDer(Q)$.

Let $m = \rk_{Z}(A) = |\loz(A)|$.
Write $\loz(A) = \{ \delta_1, \delta_1, \hdots, \delta_m\}$ where $\delta_1=0$. For each $1 \leq i \leq m$, choose $0\neq f_i \in A_{\delta_i}$. Then
\[
Q(A) = \bigoplus_{i = 1}^m Q_{\delta_i} = \bigoplus_{i = 1}^m f_i Q_0.
\]
By Lemma \ref{xxlem2.2}(5), for all $1 \leq i,j \leq m$, we have
$\{ f_i,f_j\} = q_{ij} f_if_j$ for some scalars $q_{ij}\in \kk$.
In fact, if $y_i \in Q_{\delta_i}$, then we may write $y_i = f_iz$ for some $z \in Q_0$, so $\delta_j(y_i) = \delta_j(f_i z) = q_{ij}f_i z = q_{ij} y_i$.
As a consequence, each $\delta_i$ is diagonalizable  as a linear operator on $Q(A)$. 
For all $i \neq i'$, since $\delta_i \neq \delta_{i'}$, there exists $j$ such that 
\begin{equation}
\label{E2.12.1}\tag{E2.12.1}
q_{ij} \neq q_{i'j}.
\end{equation} In particular, if $i\neq 1$, then 
there exists $j$ such that $q_{ij}\neq 0$. We now prove the following two claims.

\bigskip

\noindent
{\bf Claim 1:} Suppose $x\in Q$. Then 
$x\in Q_{\delta}$ for some $\delta\in \loz(A)$ if and only if $x$ is a common eigenvector for $\delta_j$ for all $1 \leq j \leq m$.

\noindent
{\bf Proof of Claim 1:}
If $x\in Q_{\delta_i}$, then we know $\delta_j(x) =  q_{ij} x$. So $x$ is an eigenvector of $\delta_j$ for all $j$. Conversely assume that $x$ is a common eigenvector for all of the $\delta_j$. If $x$ is not in $Q_{\delta}$
for some $\delta$, then $x=\sum_i z_i f_i$ where $z_i\in Q(Z)$ and at least two $z_i$ are nonzero. Without loss of generality, we assume that $z_{m-1}$ and $z_{m}$ are nonzero. By \eqref{E2.12.1}, $q_{m-1j}\neq q_{mj}$ for some $j$. Then $\delta_j(x)=\sum_{i=1}^{m} q_{ij}z_j f_j$ which is not in $\kk x$ since $q_{m-1j}\neq q_{mj}$, which is a contradiction. This finishes the proof of Claim 1.

\medskip

For each $1\leq i\leq m$, we define 
\[ A^{i}:=\{y\in A\mid \delta_j(y)=q_{ij} y,\; \forall \; 
\delta_j\in \loz(A)\}\]
and
\[Q^{i}:=\{y\in Q\mid \delta_j(y)=q_{ij} y,\; \forall \; 
\delta_j\in \loz(A)\}.\]
It is clear that $A^{i}=A\cap Q^{i}$ and by Claim 1 we have $Q^{i}=Q_{\delta_i}$. Consequently,
$A_{\delta_i}\subseteq A^{i}$ for all $i$ and
$\sum_{i=1}^{m} A^{i}=\bigoplus_{i=1}^m A^{i}$. 

\bigskip
\noindent
{\bf Claim 2:} $A=\bigoplus_{i=1}^m A^{i}$.

\noindent
{\bf Proof of Claim 2:} For each $N \geq 0$, let $V_N$ be the $\kk$-space $\bigoplus_{s=0}^{N} A_{s}$. Recall that each $\delta\in \loz(A)$ is diagonalizable and since $\delta$ is in $\PDer_0(A)$, it is diagonalizable on $V_N$. By Lemma~\ref{xxlem2.2}(5), we see that $\delta_i \delta_j = \delta_j \delta_ i$ for all $1\leq i,j\leq m$.
Since $\kk$ is algebraically closed, the linear transformations $\delta_i: V_N \to V_N$ are simultaneously diagonalizable. Hence, $V_N$ is a direct sum of subspaces consisting of 
common eigenvectors of $\delta\in \loz(A)$. By Claim 1, $V_N = \bigoplus_{i=1}^m V_N \cap Q^{i}=\bigoplus_{i=1}^m V_N \cap A^{i}$. By letting $N\to \infty$, we have 
$A=\bigoplus_{i=1}^m A^{i}$. This finishes the proof of Claim 2.

\bigskip

Let $N=1$ in the proof of Claim 2. Then $A_1=\bigoplus_{i=1}^{m} A_1\cap A^{i}$ and we choose a basis of $A_1$,
say $\{x'_1,\cdots, x'_n\}$,
using elements in $A_1\cap A^{i}$ for all $i$. Then $\{x'_s, x'_t\}=q_{ij} x'_s x'_t$ if $x'_s\in A_1\cap A^{i}$ and $x'_t\in A_1\cap A^{j}$. Thus $A$ is isomorphic to $P_{\bc}$ for some
$\bc$. 
\end{proof}

\begin{remark}
\label{xxrem2.13}
Suppose $P$ is a polynomial Poisson algebra.
\begin{enumerate}
\item[(1)]
We can define $C_{\loz}(P)$ to be the intersection
\[C_{\loz}(P)=\bigcap_{\delta \in \loz(P)} \ker \delta.\]
It is clear that $C_{\loz}(P)$ contains $Z(P)$. Then $\rk_{C_{\loz}(P)}(P)$ is $p^m$ which is bounded above by $\rk_{Z(P)}(P)$.
\item[(2)]
If $P=P_{\bc}$ is as defined in Definition \ref{xxdef0.1}, then 
$C_{\loz}(P)=Z(P)$ and $P=P_{\loz}$.
As a consequence
\[\rk_{Z(P)}(P)=\rk_{C_{\loz}(P)} (P)=\rk_{Z(P)}(P_{\loz})=|\loz(P)|.\]
\item[(3)]
If $P$ is not $P_{\bc}$, then four 
numbers 
\[(\rk_{Z(P)}(P),\quad \rk_{C_{\loz}(P)} (P),\quad \rk_{Z(P)}(P_{\loz}),\quad |\loz(P)|)\]
may be different. It would be interesting to understand the relationships between these four numbers.
\end{enumerate}
\end{remark}

\begin{example}
\label{xxexa2.14}
Let $P=\kk[x_1,x_2]$ with Poisson bracket given by $\{x_1,x_2\}=x_1^2$. 
By Example \ref{xxexa1.6}, $Z(P)=\kk[x_1^p,x_2^p]$, and consequently, $\rk_{Z(P)}(P)=p^2$. It follows from Lemma \ref{xxlem2.8}(2) that 
$\loz(P)=\{\delta_{x_1}\}\cong {\ZZ}/(p)$ so $|\loz(P)|=p$.
Then
\[ C_{\loz}(P) = \ker\delta_{x_1} = \kk[x_1,x_2^p].\]
Hence, $\rk_{C_{\loz}(P)} (P) = p$.

It is clear that $\sum_{i=0}^{p-1} x_1^i Z(P)$ is a direct sum $\bigoplus_{i=0}^{p-1} x_1^{i} Z(P)$
which implies that $P$ is loz-decomposable, but not quasi-inferable.
\end{example}

\begin{example}
\label{xxexa2.15}
Let $P=P_\Omega$ where $\Omega=x_1x_2(x_1+x_2)$.
Then 
\[\{x_1,x_2\}=0, \quad 
\{x_2,x_3\}=2x_1 x_2+ x_2^2, \quad
\{x_3,x_1\}=x_1^2+2x_1x_2.\]
Suppose $p>3$. Note that $Z(P)$ is generated 
by $x_1^p, x_2^p, x_3^p$, and $\Omega$. 
It is easy to check that $x_1, x_2, x_1+x_2$ are 
Poisson normal elements, and 
\begin{align*}
\delta_{x_1}: &x_1\mapsto 0, x_2\mapsto 0, 
x_3\mapsto x_1+2x_2;\\
\delta_{x_2}: &x_1\mapsto 0, x_2\mapsto 0, 
x_3\mapsto -2x_1-x_2;\\
\delta_{x_1+x_2}: &x_1\mapsto 0, x_2\mapsto 0, x_3\mapsto x_1-x_2.
\end{align*}
As a consequence,
\begin{align*}
\delta_{x_1^2 x_2}: &x_1\mapsto 0, 
x_2\mapsto 0, x_3\mapsto 3x_2;\\
\delta_{x_1x_2^2}: &x_1\mapsto 0, 
x_2\mapsto 0, x_3\mapsto -3x_1.
\end{align*}
Since $-\Omega+x_1^2x_2+x_1x_2^2=0$, it 
follows that $P_{\delta_{x_1^2x_2}}
+P_{\delta_{x_1 x_2^2}}
+P_{\delta_{-\Omega}}$
is not a direct sum where 
$P_{\delta_{-\Omega}}=P_{0}$. So $P$ is 
not loz-decomposable. 
\end{example}

\section{Skew-symmetric Poisson algebras}
\label{xxsec3}

As above, we continue to assume that
$\kk$ is a field of characteristic $p > 0$.
The prime subfield of $\kk$ is denoted by $\FF_p$.
We study the skew-symmetric Poisson algebras $\PC = \kk[x_1, \dots, x_n]$ in $n$ variables, and specifically study the case when $n = 3$.
We will show that if $\PC$ is unimodular, then $\ZPC$ is Gorenstein.
The following theorem is inspired by results from \cite{CGWZ2}.

\begin{lemma}
\label{xxlem3.1}
The skew-symmetric Poisson algebra $\PC$ is unimodular if and only if $\sum_{j=1}^{n}c_{ij}=0$ for every $i$.
\end{lemma}
\begin{proof}
The modular derivation [Definition \ref{xxdef1.4}]
is given by
\[ 
\phi(x_i) = \sum_{j=1}^n {\textstyle\dfrac{\partial}{\partial x_j}} \{x_i,x_j\}
	= \sum_{j=1}^n c_{ij} {\textstyle\dfrac{\partial }{\partial x_j}} x_i x_j= \left(\sum_{j=1}^n c_{ij} \right)x_i.
\]
This proves the result.
\end{proof}

\begin{example}
\label{xxexa3.2}
Suppose $\PC$ is a skew-symmetric Poisson algebra in three variables. Then Lemma  \ref{xxlem3.1} implies $\PC$ is unimodular if and only if the following equalities hold:
\[ c_{12} + c_{13} = c_{21} + c_{23} = c_{31} + c_{32} = 0.\]
Antisymmetry now implies that $\PC$ is unimodular if and only if
there is some $c \in \kk$ such that
\[ \bc = \begin{bmatrix}0 & c & -c \\ -c & 0 & c \\ c & -c & 0\end{bmatrix}.\]
\end{example}

We use the results from \S6 of \cite{Stanley} where simple combinatorial conditions are given for homological properties of certain subalgebras of (commutative) polynomial algebras generated by monomials. For $\bv = (v_1, \dots, v_n) \in \NN^n$, let $x^{\bv} = x_1^{v_1} \cdots x_n^{v_n}$. For each $1 \leq i \leq n$, let $\be_i$ denote the $i$th standard basis vector of $\NN^n$ (or $\FF_p^n$ or $\kk^n$), let $\balpha_i = p\be_i$, and let $\mathds{1} = \sum_{i=1}^n \be_i$..
Given a submonoid $M'$ of $\NN^n$, let $\kk[M']$ denote the subalgebra of $\kk[x_1, \ldots, x_n]$ generated by $\{ x^\bb \mid \bb \in M'\}$. We call $M'$ a \emph{simplicial submonoid} if there exist $\bb_1, \ldots, \bb_n \in M'$
such that $x^{\bb_1},\ldots,x^{\bb_n}$ is a regular sequence in $\kk[M']$.

Consider the map
\[
\varphi_{\bc}: \NN^n \to \FF_p^n \hookrightarrow \kk^n \overset{\bc \cdot}{\to} \kk^n
\]
where the first map is given by reduction modulo $p$, and the last map is given by left multiplication by the matrix $\bc$.

\begin{lemma}
\label{xxlem3.3}
The set $M = \{ \bv \in \NN^n \mid \varphi_{\bc}(\bv) = 0\}$ is a submonoid of $\NN^n$. Further, $\ZPC = \kk[M]$.
\end{lemma}
\begin{proof}
If $\bv, \bw \in M$ then $\bc (\bv + \bw) = \bc \bv + \bc \bw = 0$ and so $M$ is a submonoid of $\NN^n$. For the second claim, note that if $\bv = (v_1, \dots, v_n) \in \NN^n$, then $x^\bv$ commutes with $x_i$ if and only if
\begin{align*}
0 &= \{x_i, x^{\bv}\} = \{ x_i, x_1^{v_1}x_2^{v_2} \cdots x_n^{v_n}\} = \sum_{j = 1}^n \{x_i, x_j^{v_j}\} x_1^{v_1} \cdots \widehat{x_j^{v_j}}\cdots x_n^{v_n} \\
&= \sum_{j = 1}^n v_j \{x_i, x_j\} x_1^{v_1} \cdots x_j^{v_j-1}\cdots x_n^{v_n} = \sum_{j = 1}^n c_{ij}v_j  x_1^{v_1} \cdots x_i^{v_i+1}\cdots x_n^{v_n}.
\end{align*}
Hence, $x^{\bv} \in \ZPC$ if and only if $\{x_i, x^{\bv}\} = 0$ for all $1 \leq i \leq n$ if and only if $\bc \bv = 0$ if and only if $x^{\bv} \in \kk[M]$.
\end{proof}

\begin{lemma}
\label{xxlem3.4}
Let $M = \{ \bv \in \NN^n \mid \varphi_{\bc}(\bv) = 0\}$, $\balpha_i = p\be_i$, and $\mathds{1} = \sum_{i=1}^n \be_i$ be as defined above. Define $\Lambda = \bigoplus_{i=1}^n\balpha_i \NN$ and $B=\{\bb = (b_1, \dots, b_n) \in M \mid 0 \le b_i < p\}$.
\begin{enumerate}
\item \label{technicallemma1regularsequence}
The elements $\balpha_1,\dots,\balpha_n \in M$, and $x^{\balpha_1},\ldots, x^{\balpha_n}$ form a regular sequence in $\kk[M]$ {\rm{(}}in any order{\rm{)}}. 
\item\label{technicallemma1regularsequenceCM} $M$ can be expressed as the disjoint union $\bigcup_{\bb\in B} (\bb+\Lambda)$. In particular, since $B$ is a finite set, this is a finite disjoint union.
\item\label{technicallemma1gorenstein} If $\PC$ is unimodular, 
then the element $(p-1)\mathds{1} \in  B$,
and for all $\bb\in B$, there is a unique $\bb'\in B$ such that $\bb + \bb'= (p-1)\mathds{1}$. 
\end{enumerate}
\end{lemma}

\begin{proof}
(1) Since each $x^{\balpha_i} = x_i^p \in \ZPC$, it is clear that $\balpha_1, \dots, \balpha_n \in M$. Let $Z=\kk[M]$ and suppose there exists some $r\in Z$ such that $x_{i+1}^p r \in (x_1^p,\ldots,x_i^p) =:I$. Since $Z$ is $\ZZ^n$-graded and $I$ is a $\ZZ^n$-graded ideal, we may assume $r$ is a monomial. Then $r \in I$ by considering the $\ZZ^n$-degree. Hence $x_{i+1}^p$ is a nonzero divisor in $Z/I$, so the sequence $(x_1^p,\ldots, x_n^p)$ is regular.
    
(2) For any $\bv \in \NN^n$, it is clear that $x^\bv \in \ZPC$ if and only if $x^{\bv + \balpha_i} \in \ZPC$ for any $1 \leq i \leq n$. Hence, if $\bv \in M$, then there is some $\bb \in B$ such that $\bv \in \bb + \Lambda$. For fixed $\bb \in B$, the elements $\bb + \Lambda$ are simply those elements of $\NN^n$ which are equal to $\bb$ modulo $p$. Since none of the distinct $\bb, \bb' \in B$ are equal modulo $p$, it is clear that $\bb + \Lambda$ and $\bb' + \Lambda$ are disjoint. Hence, $M$ can be expressed as the disjoint union, as desired.
    
(3) Let $\bw = (p-1)\mathds{1}$. Assume 
$\PC$ is unimodular. By Lemma \ref{xxlem3.1}, the columns of $\bc$ sum to $0$ modulo $p$. Hence, $\bw \in B$. Given $\bb = (b_1, \dots, b_n) \in B$, we need to show that $\bw - \bb\in B$. But since $\bc \bw = \bc \bb = 0$, then $\bc(\bw - \bb) = 0$. Since $0\le b_i < p$, then $0\le (p-1)-b_i < p$, so $\bw-\bb\in B$. 
\end{proof}

\begin{theorem}
\label{xxthm3.5}
If $\PC$ is unimodular, then $\ZPC$ is Gorenstein. 
\end{theorem}

\begin{proof}
Retain the notation of Lemma~\ref{xxlem3.4}. 
By Lemma \ref{xxlem3.4}\eqref{technicallemma1regularsequenceCM} and \cite[Corollary 6.4]{Stanley}, $\kk[M]$ is Cohen--Macaulay.
In particular, the $\beta_1,\ldots, \beta_t$ appearing in \cite[Corollary 6.4]{Stanley} are exactly the elements of $B$ (this fact is not strictly needed for the rest of the proof).

Let $\beta_t = \bw = (p-1)\mathds{1} \in B$ and $\beta_1=\bzero$. We claim that it is possible to (re)label the other $t-2$ elements of $B$ such that $\beta_i+\beta_{t+1-i}=\beta_t = \bw$ for all $1\le i \le t$. 
Choose an element of $B$ not yet labelled and call it $\beta_i$. Then 
by Lemma \ref{xxlem3.4}\eqref{technicallemma1gorenstein}, we have $\beta_t-\beta_i\in B$. So let $\beta_{t+1-i}=\beta_t-\beta_i$. Proceeding in this fashion for $i=2,\ldots,\lfloor t+1-i/2\rfloor$ we will achieve what we claimed. 
By \cite[Corollary 6.5]{Stanley}, $\kk[M]$ is Gorenstein. 
\end{proof}

\begin{remark}
\label{xxrem3.6}
We can restate \cite[Corollary 6.5]{Stanley} in terms of $B$. Define a partial order on $\NN^n$ by $\bb \le \bb'$ whenever $\bb' - \bb\in \NN^n$. Then $\kk[M]$ is Gorenstein if and only if $B$ has a maximal element. In the unimodular case, this maximal element was $\bw$.
\end{remark}

The unimodular condition on a skew-symmetric Poisson algebra is not necessary for its center to be Gorenstein. 

\begin{example}
\label{xxexa3.7}
Let $0<a<p$ and set
\begin{align}\label{regularc}
    \bc = \begin{bmatrix} 0 & a & 0 \\ -a & 0 & 0 \\ 0 & 0 & 0 \end{bmatrix}.
\end{align}
Then $\ZPC$ is generated by $x_1^p, x_2^p ,x_3$,
so is regular and hence Gorenstein, but $\PC$ is not unimodular.
\end{example}

For the next results, we establish some notation. Let $M$ be as defined in Lemma~\ref{xxlem3.3} and let $\pi_j : \NN^n\to \NN$ denote the $j$th projection. Define
\[
J = \left \{ j \mid \pi_j(M) \subseteq p \ZZ \right \} \quad \text{and} \quad 
I = \{1, \dots, n\} \setminus J.
\]
Further, let $\bu = \sum_{i \in I} \be_i$.
\begin{theorem}
\label{xxthm3.8}
Retain the above notation and let $\PC$ be a skew-symmetric Poisson algebra. If there exists some $\beta \in M$ such that for any $i \in I$, $\pi_i(\beta)$ is not divisible by $p$, then $Z\PC$ is Gorenstein if and only if $\bu \in M$. 

In particular, if $I=\{1,\ldots,n\}$, and there exists $\beta \in M$ with no component divisible by $p$, then $Z\PC$ is Gorenstein if and only if $\PC$ is unimodular.
\end{theorem}

\begin{proof}
Recall the notation of Lemma \ref{xxlem3.4} where $M$, $\Lambda$ and $B$ are related as follows
\begin{align*}
    M = \bigcup_{\bb\in B} (\Lambda + \bb).
\end{align*}
By \cite[Corollary 6.5]{Stanley},
the ring $\kk[M]$ is Gorenstein if and only if $B$ has a maximal element. So it suffices to show that $\bu \in M$  
if and only if $B$ has a maximal element.
For each $1\leq s\leq p-1$, let $\beta_s\in \NN^n$ be defined by reducing each component of $s \beta$ modulo $p$. Since $\bc \beta = 0$ we see that $\bc \beta_s = 0$. Further, each component of $\beta_s$ is between $0$ and $p-1$ and hence $\beta_s \in B$.

First, suppose that $\bu\in M$. Then $(p-1)\bu\in M$ is the maximal element in $B$ since it is an element of $B$, and for every other element $\bb\in B$, we have $\pi_j(\bb)=0$ for $j\in I$ and $\pi_i(\bb)\le p-1$ for $i\in I$, so that $(p-1)\bu-\bb\ge 0$. Conversely, suppose $B$ has a maximal element $\boldm$. Since $\FF_p^{\times}$ is a group, for each component index $i\in I$, there exists an $s$ such that the $i$th component of $\beta_s$ is equal to $p-1$. This implies the $i$-th component of $\boldm$ must be $p-1$. Since this is true for all $s$, we conclude that $\boldm=(p-1)\bu$, and consequently $\bu\in M$. 

Finally if $I=\{1,\ldots,n\}$, then $\bu=(1,\ldots,1)$, and $\bu\in M$ if and only if $\PC$ is unimodular by Lemma  \ref{xxlem3.1}.
\end{proof}

The above theorem gives easily checkable conditions, as we demonstrate in the next example.

\begin{example}
Let $n = 3$ and $p = 3$. If we let
\[
\bc = \begin{bmatrix}
   0 & 1 & 1 \\ -1 & 0 & -1 \\ -1 & 1 & 0
\end{bmatrix}
\]
then since $(1,1,2) \in B$, we have that $I = \{1,2,3\}$ and there exists $\beta = (1,1,2)$ which satisfies the hypotheses of Theorem~\ref{xxthm3.8}. However, since
\[
B = \{ (0,0,0), (1,1,2), (2,2,1) \}
\]
we see that $\bu \not \in M$ and so the $\ZPC$ is not Gorenstein. Indeed, the center is generated by $x_1 x_2 x_3^2$, $x_1^2 x_2^2 x_3$, and $x_1^3, x_2^3, x_3^3$.

On the other hand, if we let
\[
\bc = \begin{bmatrix}
    0 & 1 & -1 \\ -1 & 0 & 1 \\ 1 & -1 & 0
\end{bmatrix}
\]
then $B = \{(0,0,0), (1,1,1)\}$ and since $\bu \in M$, we see that $\ZPC$ is Gorenstein. Indeed, the center is generated by $x_1 x_2 x_3$ and the $x_1^3, x_2^3, x_3^3$ so is isomorphic to $\kk[a,b,c,d]/(a^3 - bcd)$.
\end{example}

\begin{example}
The vector $\beta$ in Theorem~\ref{xxthm3.8} may or may not exist. For example, if we let $p = 3$ and let
\[
\bc = \begin{bmatrix}
    0 & 1 & -1 & -1 \\ -1 & 0 & 1 & -1 \\ 1 & -1 & 0 & -1 \\ 1 & 1 & 1 & 0
\end{bmatrix},
\]
then over $\FF_3$, the nullspace of $\bc$ is spanned by $(1,1,1,0)$ and $(1,2,0,2)$. Therefore,
\begin{align*}
B= \{&(0,0,0,0), (0,1,2,2), (0,2,1,1), (1,0,2,1), \\ 
&(1,1,1,0), (1,2,0,2), (2,0,1,2), (2,1,0,1) (2,2,2,0) \}.
\end{align*}
Hence, $I = \{1,2,3,4\}$, but there is no single vector $\beta \in M$ such that all components of $\beta$ are nonzero modulo $3$.
\end{example}

Although the vector $\beta$ in Theorem~\ref{xxthm3.8} may not exist, if $p > n$, then it is guaranteed to exist, as the next result demonstrates.

\begin{proposition}
\label{xxpro3.9}
Assume $n < p =  \ch \kk$. Then either
\begin{enumerate}
\item[(1)]
$\ZPC=\kk[x_1^{p},\cdots, x_{n}^p]$, or
\item[(2)]
$I \neq \emptyset$ and there exists some $\beta \in M$ such that for all $i \in I$, $\pi_i(\beta)$ is not divisible by $p$.
\end{enumerate}
As a consequence, $\ZPC$ is Gorenstein if and only if $\ZPC=\kk[x_1^p,\cdots,x_n^p]$
or condition (2) holds and $\bu\in M$.
\end{proposition}

\begin{proof} Suppose we are not in case (1). Then it is clear that $I$ is nonempty. Up to a permutation, 
we may assume that $I=\{1,\cdots, w\}$ where $w\leq n$. So for each $i\in I$, there exists $\bb_i \in M$ such that $\pi_i(\bb_i)$ is not divisible by $p$.

We claim that there is a $\beta \in M$ such that for all $i \in I$, $\pi_i(\beta)$ is not divisible by $p$.
Consider $M':=M/(p \ZZ)^n$ as a subspace of 
$(\ZZ/(p))^n$ and abuse the notation $\pi_i$ for the projection from 
$(\ZZ/(p))^n\to \ZZ/(p)$. We also abuse notation and view $\bb_i \in M'$. We need to find an element
in $M'$ such that its $i$th component, for 
every $i\in I$, is nonzero.
In fact, we claim that for every $1 \leq m \leq w$, there exists some $\beta_m \in M'$ such that $\pi_i(\beta_m) \neq 0$ for every $1 \leq i \leq m$. We proceed by induction on $m$. When $m = 1$, we can let $\beta_1 = \bb_1$. Now suppose that $m > 1$, and we have proved the claim for $m - 1$ and so produced $\beta_{m-1}$. Consider $\beta_{m-1} + c \bb_m$ where $c \in \ZZ/(p)$. For each $i < m$, there is at most one possible choice for $c$ such that $\pi_i(\beta_{m-1} + c \bb_m) = 0$ (if $\pi_i(\bb_m) = 0$, then there are no choices, otherwise we must choose $c = \pi_i(\bb_m)\inv\pi_i(\beta_{m-1})$, if this element is in $\ZZ/(p)$). There is also at most one possible choice for $c$ such that $\pi_m(\beta_{m-1} + c \bb_m) = 0$ (namely, $c = \pi_m(\bb_m)\inv\pi_m(\beta_{m-1})$.

Since we need to avoid only $m \leq n < p$ values of $c \in \ZZ/(p)$, therefore there exists some $c$ such that $\beta_m = \beta_{m-1} + c \bb_m$ satisfies $\pi_i(\beta_m) \neq 0$ for all $i \leq m$. The claim follows by induction. The consequence follows from Theorem \ref{xxthm3.8}.
\end{proof}

\begin{corollary}
\label{xxcor3.10}
Assume $n=3$ and $p>3$. Then $\ZPC$ is Gorenstein if and only if either
\begin{enumerate}
\item[(1)]
$\ZPC=\kk[x_1^{p},x_2^{p}, x_{3}^p]$, or
\item[(2)]
for some $a \in \kk$,
\[
\mathrm{(2a)}~
\bc \sim \begin{bmatrix} 0 & a & 0 \\ -a & 0 & 0 \\ 0 & 0 & 0 \end{bmatrix}, \quad
\mathrm{(2b)}~
\bc \sim \begin{bmatrix} 0 & a & -a \\ -a & 0 & 0 \\ a & 0 & 0 \end{bmatrix}, \quad\text{ or }~\mathrm{(2c)}~
\bc \sim \begin{bmatrix} 0 & a & -a \\ -a & 0 & a \\ a & -a & 0 \end{bmatrix}
\]
where $\sim$ denotes permutation similarity. 
\end{enumerate}
Case \textup{(2c)} is equivalent to $\PC$ being unimodular.
\end{corollary}

\begin{proof}
Suppose we are not in case (1). Then 
$I$ is nonempty, since $M$ is not in $p\NN^3$. In the case when $|I|=1$, we can assume without loss of generality that $I=\{3\}$. Then by Theorem \ref{xxthm3.8}, the center $Z\PC$ is Gorenstein if and only if $\be_3\in M$, which implies $\bc$ must be of the form (2a). In the case $|I|=2$, we can assume $I=\{2,3\}$, then the Gorenstein condition becomes $\be_2+\be_3\in M$, which implies $\bc$ must be of the form (2b). Finally, in the case $|I|=3$, the Gorenstein condition is $\be_1+\be_2+\be_3\in M$, which implies $\bc$ has the form (2c). 
\end{proof}

\section{Centers and log-ozone groups of unimodular polynomial Poisson algebras in dimension 3}
\label{xxsec4}

In this section we consider an extension of 
the questions posed in the previous section 
for unimodular polynomial Poisson algebras 
$\kk[x_1, x_2, x_3]$ in dimension $n = 3$.

\begin{lemma}
\label{xxlem4.1}
Assume $p>3$. Suppose $P = \kk[x_1,x_2,x_3]$ 
is a unimodular graded Poisson algebra. Then 
there is some $\Omega \in \kk[x_1,x_2,x_3]$ 
such that $P$ has Jacobian structure.

If $\kk$ is algebraically closed then, up to 
a linear change of variable, the polynomial 
$\Omega$ is one of the following:
\begin{align*}
&\text{(product of linears)} &  &x_1^3, \quad 
x_1^2x_2, \quad 
2x_1x_2x_3, \quad
x_1^2x_2 + x_1x_2^2, 
\\
&\text{(product of linear and quadratic)} & 
&x_1^3+x_1^2x_2 + x_1x_2x_3, \quad 
x_1^2x_3 + x_1x_2^2, \\
&\text{(irreducible)} & 
&x_1^3 + x_2^2x_3, \quad 
x_1^3 + x_1^2x_3 + x_2^2x_3, \\
& & &\frac{1}{3}(x_1^3+x_2^3+x_3^3)
+\lambda x_1x_2x_3~(\lambda^3 \neq -1).
\end{align*}
\end{lemma}

\begin{proof}
The proof that the bracket has Jacobian structure 
follows analogously as in the characteristic 
zero case assuming $p>2$ \cite[Theorem 5]{PRZ} 
(see also \cite[Proposition 2.6]{LWW}). Similarly, 
the classification of potentials, assuming $p>3$ 
and $\kk$ algebraically closed, may be adapted 
directly from the characteristic zero case
(see \cite{BM,KM,Ri}). 
\end{proof}

\subsection{Poisson centers of unimodular 
Poisson algebras of dimension 3}
\label{xxsec4.1}

The Poisson centers of unimodular 
$3$-dimensional polynomial Poisson algebras were 
computed by Etingof and Ginzburg in 
\cite[\S4]{Etingof2010} over an algebraically 
closed base field of characteristic zero. The 
positive characteristic version of this result 
is more interesting, and requires the use of different 
arguments. 

Let $\Omega$ be one of the forms in 
Lemma~\ref{xxlem4.1} and set $P=P_\Omega$. For 
any $f,g\in P$, we can write the Poisson bracket as
\[
\{g,f\} = \frac{\mathd g\wedge
\mathd f\wedge \mathd\Omega}{\mathd x_1
\wedge\mathd x_2\wedge\mathd x_3}.
\]
Then $f\in \pcnt(P)$ if and only if 
\begin{align}\label{eq.wedge}
\mathd f\wedge\mathd \Omega= 0.
\end{align}

Write $f_i = \partial f / \partial x_i$ and $\Omega_i = \partial \Omega / \partial x_i$ for $i = 1, 2, 3$. Note that
\[ \mathd f = f_{1} \mathd x_1 
+ f_{2} \mathd x_2 + f_{3} \mathd x_3 \quad \text{and} \quad  \mathd \Omega = \Omega_{1} \mathd x_1 
+ \Omega_{2} \mathd x_2 + \Omega_{3} \mathd x_3  .
\]
Taking wedge products, we see that \eqref{eq.wedge} holds if and only if the 
following equations hold:
\begin{align}\label{eq.cnt_eqs}
0 	= f_1 \Omega_{2} - f_2\Omega_{1} 
	= f_1 \Omega_{3} - f_3\Omega_{1}
	= f_3 \Omega_{2} - f_2\Omega_{3}.
\end{align}
We now use this along with some other ideas to 
compute the center of $P_\Omega$ for the forms in 
Lemma \ref{xxlem4.1}.

\begin{lemma}\label{xxlem4.2}
Assume $p>3$.
Let $\Omega$ be one of the forms in Lemma 
\ref{xxlem4.1}. If $\Omega = x_1^3$, then 
$\pcnt(P_\Omega) = \kk[x_1,x_2^p,x_3^p]$. 
Otherwise, $\pcnt(P_\Omega) = 
\kk[x_1^p,x_2^p,x_3^p,\Omega]$. In all 
cases, $Z(P_\Omega)$ is Gorenstein.
\end{lemma}
\begin{proof}
In each case, assume that $f \in Z(P_\Omega)$. 

\vspace{1em}

\noindent{\bf Case $\Omega=x_1^3$}. 
The equations \eqref{eq.cnt_eqs} in this case 
become $3x_1^2 f_2 = 3x_1^2 f_3=0$. Hence, 
$f \in \pcnt(P_\Omega)$ if and only if the 
degree of $x_2$ and $x_3$ in any term is 
divisible by $p$. Since $x_1 \in \pcnt(P_\Omega)$, 
then the result follows.

\vspace{1em}

\noindent{\bf Case $\Omega = x_1^2x_2$}. 
The equations \eqref{eq.cnt_eqs} in this case 
become $2x_1x_2 f_3 = x_1^2f_3 = 0$ and 
$x_1^2 f_1 = 2x_1x_2 f_2$. The first equations 
imply that the $x_3$ degree of each summand of 
$f$ must be divisible by $p$. The second 
equation implies that the $x_1$ degree of each 
summand of $f$ must be twice its $x_2$ degree 
modulo $p$. This proves the result.

\vspace{1em}

\noindent{\bf Case $\Omega = 2x_1x_2x_3$}. 
In this case, $P_{\Omega}$ is skew-symmetric, and 
we are in case (2c) of Corollary~\ref{xxcor3.10}. It is 
then easy to compute that $\pcnt(P_{\Omega}) = \kk[x_1^p,x_2^p,x_3^p,x_1x_2x_3]$.

\vspace{1em}

\noindent{\bf Case $\Omega = x_1^2x_2 + x_1x_2^2$}.
The bracket is given by
\[ \{x_1,x_2\} = 0 \qquad 
\{x_2,x_3\} = 2x_1x_2 + x_2^2 \qquad 
\{x_3,x_1\} = x_1^2 + 2x_1x_2.
\]
This is a Poisson Ore extension 
$\kk[x_1,x_2][x_3;\beta]$ where
$\beta(x_2)=2x_1x_2 + x_2^2$ and 
$\beta(x_1)=-(x_1^2 + 2x_1x_2)$.
Since $\gcd(\beta(x_2),\beta(x_3))=1$, then by 
\cite[Proposition 4.4]{Jconst2},
$\kk[x_1,x_2]^\beta = \kk[x_1^p,x_2^p,\Omega]$. 
The result now follows from Lemma \ref{xxlem1.5}(2).

\vspace{1em}

\noindent{\bf Case $\Omega = x_1^3+x_1^2x_2 + x_1x_2x_3$}.
We replace $\Omega$ with equivalent potential 
$x_1^3 + x_1x_2x_3 = x_1\gamma$ where 
$\gamma = x_1^2+x_2x_3$. Then the bracket is 
given by
\[ \{x_1,x_2\} = x_1x_2 \qquad 
\{x_2,x_3\} = x_2x_3 + 3x_1^2 \qquad 
\{x_3,x_1\} = x_1x_3.
\]
Note that $\alpha(-)=x_1\inv \{x_1,-\}$ 
is a derivation on $\kk[x_2,x_3]$. Since $\gcd(\alpha(x_2),\alpha(x_3))=1$, then again by 
\cite[Proposition 4.4]{Jconst2}
we have $\kk[x_2,x_3]^\alpha = \kk[x_2^p,x_3^p,x_2x_3]$. 
Let $f \in \pcnt(P_\Omega)$. Then by the above, $f=\sum c_{ij} (x_2x_3)^i x_1^j$. Write $x_2x_3=\gamma-x_1^2$, then expanding $f$ gives $f=\sum c_{ij}' \gamma^i x_1^j$. Thus,
\[ 
    \{x_2,f\} = \sum c_{ij}' \gamma^i (i-j) \gamma^i x_1^j x_2.
\]
Thus, if $c_{ij}'\neq 0$, then $i \equiv j \mod p$. That is, $f \in \kk[x_1^p,x_2^p,x_3^p,\Omega]$.

\vspace{1em}

\noindent{\bf Case $\Omega = x_1^2x_3 + x_1x_2^2$}.
The bracket is given by 
\[ \{x_1,x_2\} = x_1^2 \qquad 
\{x_2,x_3\} = 2x_1x_3 + x_2^2 \qquad 
\{x_3,x_1\} = 2x_1x_2.
\]
We pass to $P_\Omega[x_1\inv]$ and replace $x_3$ with $\widehat{x_3}=x_3+x_1\inv x_2^2$. The bracket is then given by
\[ \{x_1,x_2\} = x_1^2 \qquad 
\{x_2,\widehat{x_3}\} = 2x_1\widehat{x_3} \qquad 
\{\widehat{x_3},x_1\} = 0.
\]
Hence, $\beta(-)=\{x_2,-\}$ is a derivation on $\kk[x_1,\widehat{x_3}]$. Let $\sum c_{ij} x_1^i \widehat{x_3}^j \in \kk[x_1,\widehat{x_3}]^\beta$. Then a computation shows that
\[ 
0 = \beta\left(\sum c_{ij} x_1^i \widehat{x_3}^j\right) = \sum c_{ij} (2j-i) x_1^{i+1} \widehat{x_3}^j,
\]
so $i \equiv 2j \mod p$. Note that $x_1^2\widehat{x_3} = \Omega$.
It follows from Lemma \ref{xxlem1.5}(2) that
$\pcnt(P_\Omega[x_1\inv]) = \kk[x_1^p,x_2^p,\widehat{x_3}^p,\Omega]$.
Since $\pcnt(P_\Omega) = \pcnt(P_\Omega[x_1\inv]) \cap P_\Omega$, it follows that $\pcnt(P_\Omega) = \kk[x_1^p,x_2^p,x_3^p,\Omega]$.

\vspace{1em}

\noindent{\bf Case $\Omega = x_1^3 + x_2^2x_3$}.
The bracket is given by
\[ \{x_1,x_2\} = x_2^2 \qquad 
\{x_2,x_3\} = 3x_1^2 \qquad 
\{x_3,x_1\} = 2x_2x_3.
\]
We pass to $P_\Omega[x_2\inv]$ and replace $x_3$ with $\widehat{x_3}=x_3+x_2^{-2}x_1^3$. The bracket is then given by
\[ 
    \{x_1,x_2\} = x_2^2 \qquad 
    \{x_2,\widehat{x_3}\} = 0 \qquad 
    \{\widehat{x_3},x_1\} = 2x_2\widehat{x_3}.
\]
We observe that $P_\Omega[x_2\inv]$ is isomorphic to $P_{\Omega'}[x_1\inv]$ where $\Omega'=x_1^2x_3 + x_1x_2^2$ as in the previous case.
The map $P_\Omega[x_2\inv] \to P_{\Omega'}[x_1\inv]$ is given by $x_1 \mapsto -x_2$, $x_2 \mapsto x_1$, and $\widehat{x_3} \mapsto \widehat{x_3}$. It follows that $\pcnt(P_\Omega[x_2\inv])=\kk[x_1^p,x_2^p,\widehat{x_3}^p,\Omega]$, so $P_\Omega = \kk[x_1^p,x_2^p,x_3^p,\Omega]$.

\vspace{1em}

\noindent{\bf Case $\Omega = x_1^3 + x_1^2x_3 + x_2^2x_3$}.
The bracket is given by
\[ \{x_1,x_2\} = x_1^2+x_2^2 \qquad 
\{x_2,x_3\} = 3x_1^2+2x_1x_3 \qquad 
\{x_3,x_1\} = 2x_2x_3.
\]
For each $i=1,2,3$, set $\beta_i(-) = \{x_i,-\}$ and $K_i=\kk[x_i]$. Note that each $\beta_i$ is a $K_i$-derivation. 

Since $\beta_1(\Omega)=0$ and $\gcd\left( \frac{\partial\Omega}{\partial x_2}, \frac{\partial\Omega}{\partial x_3}\right) =1$, then by \cite[Proposition 4.4]{Jconst2}, $K_1[x_2,x_3]^{\beta_1} = K_1[x_2^p,x_3^p,\Omega]$.
By symmetry we have
\[
\pcnt(P)=\bigcap P^{\beta_i}
    = K_1[x_2^p,x_3^p,\Omega] \cap K_2[x_1^p,x_3^p,\Omega] \cap K_3[x_1^p,x_2^p,\Omega]
    = \kk[x_1^p,x_2^p,x_3^p,\Omega].
\]

\noindent{\bf Case $\Omega = \frac{1}{3}(x_1^3+x_2^3+x_3^3)+\lambda x_1x_2x_3~(\lambda^3 \neq -1)$}. 
The proof is identical to the previous case.
\end{proof}

\begin{theorem}
\label{xxthm4.3}
Assume $\kk$ is algebraically closed and $p>3$. Let $P$ be a graded polynomial 
Poisson algebra of dimension $3$. If $P$ is unimodular, then $\pcnt(P)$ is Gorenstein.
\end{theorem}

\begin{proof}
The unimodular graded polynomial 
Poisson algebras of dimension $3$ are listed in Lemma \ref{xxlem4.1}. By Lemma \ref{xxlem4.2}, $Z(P)$ is Gorenstein.
\end{proof}

\subsection{Log-ozone groups for unimodular polynomial Poisson algebras of dimension 3}
\label{xxsec4.2}

We next classify the log-ozone groups for 
unimodular polynomial Poisson algebras of 
dimension 3. Recall the Euler derivation $E$ 
as defined in Example \ref{xxexa1.3}.

\begin{lemma}
\label{xxlem4.4}
Suppose $\kk$ is algebraically closed and 
assume $p>3$. Let $P=\kk[x_1,x_2,x_3]$ be a 
unimodular polynomial Poisson algebra with 
potential $\Omega$. If $\Omega$ is irreducible, 
then $\PDer_0(P) = \kk E$ where $E$ is the 
Euler derivation. Hence, $\loz(P) = 0$.
\end{lemma}

\begin{proof}
That $\PDer_0(P) = \kk E$ follows from 
\cite[Example 6.6]{TWZ}. By Lemma \ref{xxlem2.4} 
the log-ozone group is a subgroup of 
$\PDer_0(P)=\kk E$. If $\delta_f\in \loz(P)$, 
then there exists $\lambda\in\kk$ such that 
$\delta_f = \lambda E$. But $\delta_f(\Omega)=0$ 
so $\lambda E(\Omega)=3\lambda \Omega=0$. Since 
$p>3$, we must have $\lambda=0$, so $\delta_f=0$. 
Hence $\loz(P)=0$.  
\end{proof}

Before the next result, we give some background 
on twists of Poisson algebras from \cite{TWZ}. 
If $A$ is a graded Poisson algebra and $\delta 
\in \PDer_0(A)$, then the Poisson twist 
$A^\delta$ is $A$ as an algebra and the Poisson 
bracket $\{-,-\}_{\text{new}}$ on $A^\delta$ is 
given by
\begin{align}\label{eq.twist}
\{a,b\}_{\text{new}} 
= \{a,b\} + E(a)\delta(b) - \delta(a)E(b).
\end{align}

Recall that for a derivation $\delta$ of 
$P=\kk[x_1,x_2,x_3]$ we have
\[ \div(\delta) = \sum_{i=1}^3 
\frac{ \partial \delta(x_i)}{\partial x_i}.\]
The next lemma will be useful for proving that 
the modular derivation belongs to the log-ozone 
group for graded polynomial Poisson algebras 
in three variables.

\begin{lemma}
\label{xxlem4.5}
Let $\delta \in \PDer(P_\Omega)$.
Then 
\begin{align*}
\div(\delta)\frac{\partial \Omega}{\partial x_k} 
&= \frac{\partial \delta( \Omega)}{\partial x_k}
\end{align*}
for $k=1,2,3$. 
In particular, if $p>3$ and $\delta \in \PDer_0(P_\Omega)$, 
then $\delta(\Omega)=\div(\delta)\Omega$.
\end{lemma}

\begin{proof}
Let $A=P_\Omega$. Recall that there is a 
canonical isomorphism
\begin{align*}
\Hom_A(\Omega_{A/\kk}, M) 
&\to\Der(A,M) \\
f&\mapsto (a\mapsto f(\mathd a))
\end{align*}
for any $A$-module $M$. In the special case 
$M=A$, any derivation $\delta$ of $A$ has the 
form $\delta(a)= v\cdot \mathd a$ for some $v=v_1\partial_1+v_2\partial_2+v_3\partial_3
\in T_{A/\kk}$, where $\{\partial_i\}_{i=1}^3$ 
is the basis of $T_{A/\kk}$ dual to 
$\{\mathd x_i\}_{i=1}^3$. Since $\delta$ is a 
Poisson derivation (see Definition \ref{xxdef1.2}) 
we have 
\begin{align}\label{poissoneq}
  v\cdot \mathd (\{a,b\}) &= 
  \{v\cdot \mathd a,b\}+\{a,v\cdot\mathd b\}.
\end{align}
Writing $\div(\delta)$ and 
$\delta(\Omega)$ in terms of $v$, we get
\begin{align}\label{diveq}
\div(\delta)
&= \frac{\partial v_1}{\partial x_1}
+\frac{\partial v_2}{\partial x_2}
+\frac{\partial v_3}{\partial x_3}.
\end{align}
and
\begin{align}\label{deltaomegaeq}
\delta(\Omega)
&= v\cdot \mathd\Omega 
= v_1\frac{\partial\Omega}{\partial x_1} 
+ v_2\frac{\partial\Omega}{\partial x_2} 
+ v_3\frac{\partial\Omega}{\partial x_3}.
\end{align}
Let $(ijk)$ be a permutation of $(123)$ and 
substitute $a=x_i$, $b=x_j$ in \eqref{poissoneq}. 
Then
\begin{align*}
\{v\cdot \mathd x_i,x_j\}+\{x_i,v\cdot\mathd x_j\} 
&= \{v_i, x_j\} + \{x_i, v_j\}\\
&= \frac{(\mathd v_i \wedge \mathd x_j 
+ \mathd x_i\wedge\mathd v_j)
\wedge\mathd \Omega}{\mathd x_1\wedge\mathd x_2
\wedge\mathd x_3}\\
&= \frac{\partial v_i}{\partial x_i}
\frac{\partial \Omega}{\partial x_k} 
- \frac{\partial v_i}{\partial x_k}
\frac{\partial \Omega}{\partial x_i} 
+ \frac{\partial v_j}{\partial x_j}
\frac{\partial \Omega}{\partial x_k} 
- \frac{\partial v_j}{\partial x_k}
\frac{\partial \Omega}{\partial x_j}  \\
\overset{\eqref{diveq}}&{=} -\frac{\partial v_k}{\partial x_k}
\frac{\partial \Omega}{\partial x_k} 
- \frac{\partial v_i}{\partial x_k}
\frac{\partial \Omega}{\partial x_i} 
- \frac{\partial v_j}{\partial x_k}
\frac{\partial \Omega}{\partial x_j} 
+ \div(\delta)
\frac{\partial \Omega}{\partial x_k}
\end{align*}
and
\begin{align*}
v\cdot \mathd (\{x_i,x_j\}) 
&= v\cdot \mathd \frac{\partial\Omega}{\partial x_k}\\
&= v_1\frac{\partial^2\Omega}{\partial x_1\partial x_k} 
+ v_2\frac{\partial^2\Omega}{\partial x_2\partial x_k} 
+ v_3\frac{\partial^2\Omega}{\partial x_3\partial x_k} \\
\overset{\eqref{deltaomegaeq}}&{=} \frac{\partial}{\partial x_k} (\delta(\Omega)) 
- \left(\frac{\partial v_1}{\partial x_k} 
\frac{\partial \Omega}{\partial x_1} 
+ \frac{\partial v_2}{\partial x_k} 
\frac{\partial \Omega}{\partial x_2} 
+ \frac{\partial v_3}{\partial x_k} 
\frac{\partial \Omega}{\partial x_3} \right) \\
\end{align*}
are equal, so we obtain the required equality. 

If $\delta \in \PDer(P_\Omega)$, then 
$\div(\delta)\in\kk$ and 
\begin{align*}
\frac{\partial (\div(\delta)
\Omega - \delta(\Omega))}{\partial x_k} 
&= 0
\end{align*}
for $k=1,2,3$. This implies $\delta(\Omega) 
- \div(\delta)\Omega \in \kk[x^p,y^p,z^p]$, 
so $\delta(\Omega)=\div(\delta)\Omega$ as $p>3$. 
\end{proof}

\begin{lemma}
\label{xxlem4.6}
Assume $p > 3$. Let $P=\kk[x_1,x_2,x_3]$ be a 
graded Poisson algebra.
\begin{enumerate}
\item[(1)] 
If $P$ is not unimodular, then the modular 
derivation $\phi$ of $P$ is in $\loz(P)$ and 
so $\loz(P) \neq 0$.
\item[(2)] 
If $P$ is unimodular and its potential $\Omega$ 
is either irreducible or equal to $x_1^3$, then 
$\loz(P)=0$.
\item[(3)] 
If $P$ is unimodular and its potential $\Omega
\neq x_1^3$ is reducible, then $\loz(P) \neq 0$.
\end{enumerate}
\end{lemma}
\begin{proof}
(1)  Let $\phi$ be the modular derivation of $P$, 
which is nonzero. Set $\delta=\frac{1}{3}\phi$ 
and note that the sign difference from \cite{TWZ} 
because we have $\phi(a)=\div(H_a)$. By 
\cite[Corollary 0.3]{TWZ}, $P^\delta$ is a 
unimodular Poisson algebra and so there is a 
potential $\Omega$ such that the bracket 
$\{-,-\}_{\text{new}}$ on $P^\delta$ is given 
by \eqref{eq.twist}. Explicitly, for 
$1 \leq i,j \leq 3$,
\begin{equation}
\notag
\{ x_i,x_j \}_{\text{new}}
= \{ x_i,x_j \}_P + x_i\delta(x_j) - x_j\delta(x_i).
\end{equation} 
Note that $\delta \in \PDer_0(P)$ and 
also $\delta \in \PDer_0(P^\delta)$. Since 
$P^\delta$ unimodular, then $P \ncong P^\delta$, 
so $\delta \notin \kk E$. Thus $\div(\delta)=0$ 
\cite[Theorem 3.8]{TWZ} and so $\delta(\Omega)=0$ 
by Lemma \ref{xxlem4.5}. Moreover, $\Omega$ is a 
central element of $P^\delta$. Thus, for any $i$,
\[
0   = \{ x_i,\Omega \}_{\text{new}}
    = \{ x_i,\Omega \}_P + x_i\delta(\Omega) 
      - 3\Omega\delta(x_i)
    = \{ x_i,\Omega \}_P - 3\Omega\delta(x_i).
\]
Thus, $\{x_i,\Omega\}_P=\phi(x_i)\Omega$. This proves 
the claim.

(2) For the irreducible cases, see Lemma 
\ref{xxlem4.4}. It is left to consider the case 
$\Omega=x_1^3$. The Poisson bracket is given, 
after scaling, as
\[ \{x_1,x_2\} = \{x_1,x_3\} = 0, 
\qquad \{x_2,x_3\} = x_1^2.\]
Recall that this is Poisson Ore extension 
$\kk[x_1,x_2][x_3;0,\beta]$ where 
$\beta(x_1)=0$ and $\beta(x_2)=x_1^2$.
To prove $\loz(A_\Omega)=0$, it suffices to 
prove that there are no noncentral Poisson normal 
elements. To see this, let $a\in P\setminus Z(P)$. 
Since $Z(P)=\kk[x_1,x_2^p,x_3^p]$, without loss 
of generality, we can write $a=\sum c_i x_2^i$ 
where $c_i\in \kk[x_1,x_3]$ for all $i$ and 
$c_i\neq 0$ for some $i\not\equiv 0 \mod p$. 
But then
\[ \{a,x_3\} = \sum c_i i x_1^2 x_2^{i-1}
=x_1^2 \sum c_i i x_2^{i-1}\]
and it is clear that $a$ does not divide 
the right-hand side.

(3) If $\Omega$ is reducible and not $x_1^3$, 
then it must have a degree 1 factor $y$ 
such that $\Omega=yb$ and $y\nmid b$. Since 
$y$ has degree $1$, we can assume without 
loss of generality that $y=x_1$. We claim 
that $x_1$ is Poisson normal. For any $a\in P$, 
we have 
\begin{align*}
\{a,x_1\}(\mathd x_1\wedge\mathd x_2\wedge 
\mathd x_3)
&=\mathd a\wedge\mathd x_1\wedge \mathd (x_1b)\\
&= \mathd a\wedge\mathd x_1\wedge 
(x_1\mathd b+b\mathd x_1) \\
&= x_1\left(\mathd a\wedge\mathd x_1\wedge 
\mathd b\right). 
\end{align*}
Hence $\{a,x_1\}\in x_1 P$, so $x_1$ is 
Poisson normal. Since $x_1\nmid b$, 
$\mathd a\wedge\mathd x_1\wedge \mathd b$ is 
not identically zero. This means that $x_1$ 
is not Poisson central.
\end{proof}

The following is now clear from the previous lemma.

\begin{theorem}
\label{xxthm4.7}
Assume $p>3$. Let $P = \kk[x_1, x_2, x_3]$ be 
a graded polynomial Poisson algebra. Then 
$\loz(P) = 0$ if and only if $P$ is unimodular 
and isomorphic to $P_{\Omega}$ where $\Omega$ 
is irreducible or $\Omega = x_1^3$.
\end{theorem}

We conclude this section by computing 
explicitly the log-ozone groups for unimodular 
polynomial Poisson algebras in dimension 3 with 
reducible potential.

\begin{lemma}
\label{xxlem4.8}
Suppose $\kk$ is algebraically closed. Let 
$P$ be a graded polynomial Poisson algebra. 
If $\loz(P)\cong {\ZZ}/(p)$, then $P$ 
is loz-decomposable.
\end{lemma}

\begin{proof} Let $\delta$ be a generator of 
$\loz(P)$ and write $\loz(P)={\ZZ}/(p) 
\delta$. If $\delta$ has a nonzero eigenvalue, 
then $P$ is quasi-inferable and the assertion 
follows from Lemma \ref{xxlem2.10}. 

Now assume that all eigenvalues of $\delta$ 
(and any element in $\loz(P)$) are 0. We 
continue to use the convention introduced in 
the proof of Lemma \ref{xxlem2.10}(2) and 
suppose to the contrary that $P$ is not 
loz-decomposable. Choose a minimal $n$ such 
that $\delta_1, \dots, \delta_n \in \loz(P)$ 
and nonzero $f_i \in P_{\delta_i}$ so that 
$\sum_{i = 1}^n f_i = 0$. Further we can 
assume that $f_n\in Z(P)$ after replacing 
$f_n$ by $f_n^{p}$. Since $\kk$ is algebraically 
closed,  there are $x_1$ and $x_2\in P_1$ (the 
subspace of degree 1) such that $\delta(x_1)=0$ 
and $\delta(x_2)=x_1$. Since every $\delta_i$ 
is a scalar multiple of $\delta$, we have 
$\delta_i(x_1) =0$ and $\delta(x_2)=x_1$.  Then
\[
0 = \left\{ -f_n, x_2\right\} 
= \sum_{i=1}^{n-1} \left\{ f_i, x_2 \right\} 
= \sum_{i=1}^{n - 1} f_i \delta_i(x_2) 
= \sum_{i=1}^{n - 1} f_i x_1.
\]
Since $P$ is a domain, we may cancel $x_1$ and 
so $\sum_{i=1}^{n-1} f_i=0$. This contradicts 
the minimality of $n$.
\end{proof}

The following result is a straightforward 
case-by-case computation. The details are omitted.

\begin{proposition}
\label{xxpro4.9}
Assume that $p>3$ and that $\kk$ is algebraically 
closed. Let $P=\kk[x_1,x_2,x_3]$ be a unimodular 
polynomial Poisson algebra with potential $\Omega$. 
Suppose $\Omega$ is reducible, $\Omega \neq x_1^3$.
\begin{enumerate}
\item 
\label{red1} If $\Omega=x_1^2x_2$, then 
$\loz(P_\Omega) = \grp{\delta_{x_1},\delta_{x_2}} = \grp{\delta_{x_1}}
={\ZZ}/(p)$ and $P_{\Omega}$ 
is loz-decomposable, but not quasi-inferable.
\item 
\label{red2} If $\Omega=2x_1x_2x_3$, then 
$\loz(P_\Omega) 
= \grp{\delta_{x_1},\delta_{x_2},\delta_{x_3}}
=\grp{\delta_{x_1},\delta_{x_2}}
={\ZZ}/(p)^{\oplus 2}
$ and $P_{\Omega}$ is inferable.
\item 
\label{red3} If $\Omega=x_1^2x_2+x_1x_2^2$, then 
$\loz(P_\Omega) 
= \grp{\delta_{x_1},\delta_{x_2}}
={\ZZ}/(p)^{\oplus 2}$ and $P_{\Omega}$ 
is not loz-decomposable {\rm{(}}see Example \ref{xxexa5.4}(1){\rm{)}}.
\item 
\label{red4}
If $\Omega=x_1^3+x_1^2x_2+x_1x_2x_3$, then 
$\loz(P_\Omega) = \grp{\delta_{x_1}}
={\ZZ}/(p)$ and $P_{\Omega}$ is inferable.
\item 
\label{red5}
If $\Omega=x_1^2x_3+x_1x_2^2$, then
$\loz(P_\Omega) 
= \grp{\delta_{x_1}}={\ZZ}/(p)$ and 
$P_{\Omega}$ is loz-decomposable, but not 
quasi-inferable. 
\end{enumerate}
\end{proposition}

\section{Comments, examples, and questions}
\label{xxsec5}

We end by presenting a few examples, some open 
questions, and future directions for research
related to log-ozone groups and polynomial Poisson algebras.
One of the basic questions towards the structure of polynomial Poisson algebras is the following.

\begin{question}
\label{xxque5.6}
Let $P$ be a graded polynomial Poisson algebra. 
Under what conditions is $Z(P)$ (resp. 
$C_{\loz}(P)$, $P_{\loz}$) regular, 
Gorenstein, or Cohen--Macaulay? 
\end{question}

In confronting Question \ref{xxque5.6} in the unimodular case, we have throughout much of our work assumed that the characteristic $p$ of the field $\kk$ satisfies $p > 3$. This hypothesis is crucial as the next example shows.

\begin{example}
\label{xxexa5.4}
Let $P=P_\Omega$ where $\Omega=x_1x_2(x_1+x_2)$.
Then 
\[\{x_1,x_2\}=0, \quad 
\{x_2,x_3\}=2x_1 x_2+ x_2^2, \quad
\{x_3,x_1\}=x_1^2+2x_1x_2.\]

By Example~\ref{xxexa2.15}, when $p>3$, then $Z(P)$ is Gorenstein.
Now suppose $p=3$. We claim $Z(P)$ is generated 
by $x_1^3, x_2^3, x_3^3, x_1^2 x_2, x_1x_2^2$.
It is clear that $x_1^3,x_2^3,x_3^3 \in Z(P)$. 
By Lemma \ref{xxlem1.5}(2), it suffices to 
compute $\kk[x_1,x_2]^\beta$ where $\beta$ is 
the derivation induced by $x_3$. That is, $\beta(x_2)=x_2^2+2x_1x_2$ and 
$\beta(x_1)=-(x_1^2+2x_1x_2)$.
Let $\sum c_{ij} x_1^ix_2^j \in \kk[x_1,x_2]^\beta$. Then
\begin{align*}
\beta\left(\sum c_{ij} x_1^ix_2^j\right)
    &= \sum c_{ij} \left( j x_1^i x_2^{j-1} (x_2^2+2x_1x_2)- i x_1^{i-1}x_2^j (x_1^2+2x_1x_2) \right) \\
    &= \sum c_{ij} x_1^i x_2^j \left( (i-2j)x_1 + (j-2i)x_2  \right) \\
    &= \sum c_{ij} x_1^i x_2^j (i+j) (x_1+x_2),
\end{align*}
where the last step follows because we are working mod $3$. Hence, $c_{ij}\neq 0$ implies $i+j \equiv 0 \mod 3$. The result follows.

To compute the Hilbert series of $Z(P)$, it 
suffices to compute the Hilbert series of 
the subalgebra of $\kk[x_1,x_2]$ generated 
by $x_1^3,x_1^2x_2,x_1x_2^2,x_2^3$. This is 
the 3-Veronese $\kk[x_1,x_2]_{(3)}$, which 
is $\frac{1+2 t^3}{(1-t^3)^2}$. Hence, $Z(P)$ 
has Hilbert series $\frac{1+2 t^3}{(1-t^3)^3}$, 
which is not palindromic, so $Z(P)$ is not 
Gorenstein.
\end{example} 

The previous example inspires the following question.

\begin{question}
\label{xxque5.7}
Let $P=\kk[x_1,x_2,x_3]$ be a graded Poisson 
algebra over a field $\kk$ with $\ch(\kk) = 2,3$. 
Under what conditions is $Z(P)$ Cohen--Macaulay or 
Gorenstein? One could also ask for the Poisson 
center when the generators are not all given 
degree 1.
\end{question}

Other than the case of skew-symmetric Poisson algebras,
we have not attempted to classify the centers of graded
Poisson algebras outside of the unimodular setting.
The following example suggests such a classification is possible.

\begin{example}
\label{xxexa5.5}
Let $d,c \in \kk$ and let $P=\kk[x_1,x_2,x_3]$ 
with a Poisson structure given by
\[
    \{x_1,x_2\}=cx_1^2+dx_3^2 \qquad
    \{x_2,x_3\}=(2c+1)x_1x_3 \qquad
    \{x_3,x_1\}=0.
\]
Then
\begin{align*}
    \{x_1,x_1^i x_2^j x_3^k\} &= jc x_1^{i+2}x_2^{j-1}x_3^k + jd x_1^i x_2^{j-1} x_3^{k+2}\\
    \{x_2,x_1^i x_2^j x_3^k\} 
        &= (k(2c+1) - ic) x_1^{i+1}x_2^jx_3^{k} - id x_1^{i-1} x_2^j x_3^{k+2}\\ 
    \{x_3, x_1^i x_2^j x_3^k\} &= -j(2c + 1) x_1^{i + 1} x_2^{j-1} x_3^{k+1}
\end{align*}

Assume $d=0$ and $c,2c+1\ne 0$ in $\kk$, then the Poisson center $Z$ is the subalgebra of $P$ generated by the elements $x_1^p,x_2^p,x_3^p$ and $(x_1x_3^\lambda)^i$ where $\lambda = c(2c+1)^{-1}$ and $i=1,\ldots,p-1$. 
We can exhibit $X:=\Spec(Z)$ as a toric variety as follows. Let $U$ be the open subset of $X$ obtained by removing the closed subsets $(x_1^p), (x_2^p)$ and $(x_3^p)$, then 
\[ \cO(U)=\kk[x_1^{\pm p}, x_2^{\pm p}, (x_1^{1-p}x_3^\lambda)^{\pm 1}]\]
so $U$ is an open dense subset of $X$ which is isomorphic to a torus.

We can use Stanley's theorem \cite[Corollary 6.4]{Stanley} to show that $X$ is Gorenstein if and only if $c=-1$. Omitting the $y$-component, the $\beta_i$'s in that theorem are just $B:=\{(i,\lambda i \mod p)\}_{i=1}^{p-1}$. Then $X$ is Gorenstein if and only if $B$ has a maximal element, this happens if and only if $(p-1)\lambda=p-1$, iff $\lambda=1$ iff $c=-1$. 
\end{example}

In this paper we were mainly interested in 
\emph{graded} Poisson algebras. However, many  
of our definitions apply to general (polynomial) Poisson 
algebras. The following example shows that the non-graded
case may exhibit different behavior from the graded case.

\begin{example}
\label{xxexa5.1}
Let $\{a_i\}_{i=1}^{n}$ be a set of distinct 
nonzero scalars. Let $P$ be the polynomial 
Poisson algebra $\kk[x_1, x_2]$ with Poisson 
bracket determined by
\[\{x_1, x_2\}=\prod_{s=1}^{n} (x_1+a_s x_2).\]
Let $y_i=x_1+a_i x_2$ and $\delta_i=\delta_{y_i}$ 
for all $i$ . Then $y_i$ is a Poisson normal 
element and $\delta_i: x_1\mapsto a_i y_i^{-1}
\prod_{s=1}^{n} y_s, x_2\mapsto -y_i^{-1}
\prod_{s=1}^{n} y_i$. Then 
$\delta_1, \cdots,\delta_n$ are 
${\ZZ}/(p)$-linearly independent. As 
a consequence, $|\loz(P)|\geq p^n$. But 
$\rk_{Z(P)}(P)=p^2$ since $Z(P)=
\kk[x_1^p, x_2^p]$.
\end{example}

Motivated by the above example and Lemma 
\ref{xxlem2.11}(2), we ask the following
question.

\begin{question}
\label{xxque5.2}
Let $P$ be a graded polynomial Poisson 
algebra in $n$-variables. Is it true that
$|\loz(P)|\leq \rk_{Z(P)}(P)$?
\end{question}

Another family of non-graded (filtered) polynomial Poisson  
algebras arise in Lie theory.
Given a Lie algebra $\fg$, then $S(\fg)$ is a 
Poisson algebra with \emph{Kostant--Kirillov 
bracket} defined by $\{a,b\} = [a,b]$ for all 
$a,b \in \fg$.

\begin{question}
\label{xxque5.8}
Let $\fg$ be a 2- or 3-dimensional Lie algebra. 
What is the Poisson center of $S(\fg)$ and the corresponding log-ozone group?
\end{question}

We end with two questions which may suggest further directions of study.
For a Poisson algebra $P$, let $\Aut(P)$ denote the group of 
automorphism of $P$. We have the following interesting
action of $\Aut(P)$ on $\loz(P)$.

\begin{question}
\label{xxque5.3}
For $\sigma\in \Aut(P)$, 
define $\Pi_{\sigma}:\loz(P) \to \loz(P)$
by $\delta_{f}\mapsto 
\delta_{\sigma(f)}$. It is easy to see 
that $\Pi_{\sigma}$ a group automorphism 
of $\loz(P)$, so $\Aut(P)$ acts on 
$\loz(A)$. What can we say about this action?
\end{question}

In Lemma \ref{xxlem4.5} we showed that for a 
graded Poisson derivation $\delta$ of 
$P=P_\Omega$, $\delta(\Omega)=\div(\delta)\Omega$. 
This is similar to a result for $m$-Koszul 
Artin--Schelter regular algebras by Mori and 
Smith \cite[Theorem 1.2]{MS} where $\div$ 
plays the role of homological determinant. 
Therefore, we ask whether this holds more 
generally.

\begin{question}
\label{xxque5.9}
Let $P$ be a polynomial Poisson algebra with 
modular derivation $\phi$. Is there a degree 
$n$ polynomial $\Omega$ such that 
$\phi=\delta_\Omega$? That is, do we have 
$\phi \in \loz(P)$ in general?
\end{question}


\providecommand{\bysame}{\leavevmode\hbox to3em{\hrulefill}\thinspace}
\providecommand{\MR}{\relax\ifhmode\unskip\space\fi MR }
\providecommand{\MRhref}[2]{%
  \href{http://www.ams.org/mathscinet-getitem?mr=#1}{#2}
}
\providecommand{\href}[2]{#2}

\end{document}